\documentclass[a4paper]{amsart}
\usepackage{amsmath}
\usepackage{amssymb}
\usepackage{amsthm,amsxtra}
\usepackage{cite}
\usepackage{enumerate}
\usepackage{enumitem}
\usepackage[mathscr]{eucal}
\usepackage{adjustbox}
\usepackage{float}
\usepackage{tikz-cd}
\usepackage{tkz-graph}
\usetikzlibrary{positioning,arrows,shapes.geometric,trees}
\usepackage{graphicx}
\usepackage{mathtools}
\usetikzlibrary{positioning}
\usetikzlibrary{decorations.text}

\usepackage{geometry}
\newgeometry{vmargin={20mm}, hmargin={25mm,25mm}}
\newtheorem{Def}{Definition}[section]
\newtheorem{Th}[Def]{Theorem}
\newtheorem{Ex}[Def]{Example}
\newtheorem{Lemma}[Def]{Lemma}
\newtheorem{Prop}[Def]{Proposition}
\newtheorem{Cor}[Def]{Corollary}

{}
\newtheorem{Prob}[Def]{Problem}
\newtheorem{Conj}[Def]{Conjecture}
\newtheorem{Res}[Def]{Result}

\DeclareMathOperator{\Sf}{S_{fin}}
\DeclareMathOperator{\Gf}{G_{fin}}
\DeclareMathOperator{\S1}{S_1}
\DeclareMathOperator{\G1}{G_1}
\DeclareMathOperator{\Uf}{U_{fin}}
\DeclareMathOperator{\Guf}{G_{ufin}}
\DeclareMathOperator{\SS1}{S_1^\ast}
\DeclareMathOperator{\SSS1}{SS_1^\ast}
\DeclareMathOperator{\SG1}{G_1^\ast}
\DeclareMathOperator{\SSG1}{SG_1^\ast}
\DeclareMathOperator{\SSf}{S_{fin}^\ast}
\DeclareMathOperator{\SGf}{G_{fin}^\ast}
\DeclareMathOperator{\SGuf}{G_{ufin}^\ast}
\DeclareMathOperator{\SUf}{U_{fin}^\ast}
\DeclareMathOperator{\SSSf}{SS_{fin}^\ast}
\DeclareMathOperator{\SSGf}{SG_{fin}^\ast}

\DeclareMathOperator{\Int}{Int}
\DeclareMathOperator{\Cl}{Cl}

\DeclareMathOperator{\cof}{cof}
\DeclareMathOperator{\Fin}{Fin}

\DeclareMathOperator{\maxfin}{\sf maxfin}

\DeclareMathOperator{\NSUf}{^\ast U_{fin}}
\DeclareMathOperator{\NSU1}{^\ast U_1}
\DeclareMathOperator{\NSSUf}{^\ast SU_{fin}}
\DeclareMathOperator{\NSSU1}{^\ast SU_1}
\DeclareMathOperator{\NSG1}{^\ast G_1}
\DeclareMathOperator{\NSGuf}{^\ast G_{ufin}}
\DeclareMathOperator{\RGuf}{RG_{ufin}}
\DeclareMathOperator{\SSRGf}{RSG_{fin}^\ast}
\DeclareMathOperator{\SRGuf}{RG_{ufin}^\ast}
\begin{document}
\title[On certain star versions of the Scheepers property]{On certain star versions of the Scheepers property}

\author[ D. Chandra, N. Alam ]{ Debraj Chandra$^*$, Nur Alam$^*$ }
\newcommand{\acr}{\newline\indent}
\address{\llap{*\,}Department of Mathematics, University of Gour Banga, Malda-732103, West Bengal, India}
\email{debrajchandra1986@gmail.com, nurrejwana@gmail.com}

\thanks{ The second author
is thankful to University Grants Commission (UGC), New Delhi-110002, India for granting UGC-NET Junior Research Fellowship (1173/(CSIR-UGC NET JUNE 2017)) during the tenure of which this work was done.}


\maketitle

\begin{abstract}
The star versions of the Scheepers property, namely star-Scheepers, strongly star-Scheepers and new star-Scheepers property have been introduced. We explore further ramifications concerning critical cardinalities. Quite a few interesting observations are obtained while dealing with the Isbell-Mr\'{o}wka spaces, Niemytzki plane and Alexandroff duplicates. The properties like monotonically normal, locally countable cellularity (which is introduced here) play an important role in our investigation. We study games corresponding to the classical and star variants of the Scheepers property which have not been investigated in prior works. Some open problems are also posed.
\end{abstract}
%\smallskip

\noindent{\bf\keywordsname{}:} {Scheepers property, star-Scheepers, strongly star-Scheepers, new star-Scheepers.}

\section{Introduction}
%The systematic study of selection principles and the corresponding games in topology was started by Scheepers \cite{coc1} (see also \cite{coc2}). In \cite{LjSM}, Ko\v{c}inac introduced star selection principles and mentioned the corresponding games. The new star selection principles introduced and studied in \cite{NSP}. The seminal papers \cite{coc1,coc2} set up a framework for studying generalization of selection principles in many ways. Readers interested in star selection principles and new star selection principles may consult the papers \cite{dcna21,dcna22,dcna22-1,LjSM,sHR,survey,SVM,NSP,NSP1,RNSP,
%FRNSP,NSL,EDSL} where more references can be found.
%Note that one of the most important selection principle is $\Uf(\mathcal{O},\Omega)$, nowadays called the Scheepers property (see \cite{coc1,coc2}). In general, Scheepers property is stronger than the Menger property \cite{coc1,coc2} and weaker than the Hurewicz property \cite{coc1,coc2}. For more information about the Scheepers property see \cite{FPM,CST}.
The systematic study of selection principles and the corresponding games in topology was initiated by Scheepers \cite{coc1} (see also \cite{coc2}) and in the last twenty five years it has become one of the most active research areas of set theoretic topology.
Various topological properties have been defined or characterized in terms of selection principles. The selection principles also have various applications in several branches of Mathematics. Generalizing the idea of selection principles, in 1999 Ko\v{c}inac \cite{LjSM} introduced star selection principles and also mentioned the corresponding games. Since then the study of star selection principles has attracted many researchers and recently a lot of investigations have been explored to enrich this area. The idea of new star selection principles was introduced in \cite{NSP} and later some fascinating investigations have been carried out in this emerging field. The seminal papers \cite{coc1,coc2} set up a framework for studying generalization of selection principles in numerous ways. Readers interested in star selection principles and new star selection principles may consult the papers \cite{dcna21,dcna22,dcna22-1,dcna23,LjSM,sHR,survey,SVM,NSP,RNSP} where more references can be found.
Note that one of the most important selection principle is $\Uf(\mathcal{O},\Omega)$, nowadays called the Scheepers property (see \cite{coc1,coc2}). In general, Scheepers property is stronger than the Menger property \cite{coc1,coc2} and weaker than the Hurewicz property \cite{coc1,coc2}. For more information about the Scheepers property see \cite{FPM,CST}.

In this paper we introduce and study the star versions of the Scheepers property $\Uf(\mathcal{O},\Omega)$, namely star-Scheepers property $\SUf(\mathcal{O},\Omega)$, strongly star-Scheepers property $\SSSf(\mathcal{O},\Omega)$ and new star-Scheepers property $\NSUf(\mathcal{O},\Omega)$. In \cite[Proposition 1.7]{SVM}, Sakai proved that every star-Lindel\"{o}f \cite{sCP} (respectively, strongly star-Lindel\"{o}f \cite{sCP}) space of cardinality less than $\mathfrak{d}$ is star-Menger \cite{LjSM} (respectively, strongly star-Menger \cite{LjSM}). We improve this result by showing that every star-Lindel\"{o}f (respectively, strongly star-Lindel\"{o}f) space of cardinality less than $\mathfrak{d}$ is star-Scheepers (respectively, strongly star-Scheepers). In \cite{SCPP}, Bonanzinga and Matveev observed that the Isbell-Mr\'owka space $\Psi(\mathcal{A})$ \cite{Mrowka} is strongly star-Menger if and only if $|\mathcal{A}|<\mathfrak{d}$. Again we give an improvement of this result by showing that $\Psi(\mathcal{A})$ is strongly star-Scheepers if and only if $|\mathcal{A}|<\mathfrak{d}$. We also show that similar result holds in the realm of the Niemytzki plane \cite{Niemytzki}.
%Similarly we also improve the result: the Niemytzki plane $N(Y)$ \cite{Niemytzki} is strongly star-Menger if and only if $|Y|<\mathfrak{d}$ (see \cite{SSSP}).

Our investigation shows that the answers to the following problems are not affirmative.
\begin{enumerate}[wide=0pt,label={\upshape(\arabic*)},leftmargin=*]
\item Is there a space $X$ such that $AD(X)$ is star-Scheepers, but $X$ is not star-Scheepers?
\item Is there a space $X$ such that $AD(X)$ is strongly star-Scheepers, but $X$ is not strongly star-Scheepers?
\item Is there a space $X$ such that $AD(X)$ satisfies $\NSUf(\mathcal{O},\Omega)$, but $X$ does not satisfy $\NSUf(\mathcal{O},\Omega)$?
\end{enumerate}
In fact our observation further indicates that the answers to the similar problems posted in \cite{rsM,rssM} (for Menger-kind) are not affirmative.
%We answer these problems for star-Scheepers, strongly star-Scheepers and $\NSUf(\mathcal{O},\Omega)$ property, and same is true for star-Menger and strongly star-Menger property.
%The following problems for the Alexandroff duplicate $AD(X)$ \cite{AD} of a space $X$ were posed by Song in \cite{rsM,rssM}.
%\begin{enumerate}[wide=0pt,label={\upshape(\arabic*)},leftmargin=*]
%\item Is there a space $X$ such that $AD(X)$ is star-Menger, but $X$ is not star-Menger?
%\item Is there a space $X$ such that $AD(X)$ is strongly star-Menger, but $X$ is not strongly star-Menger?
%\end{enumerate}
The following is a summary of what has been done throughout.
\begin{enumerate}[wide=0pt,label={\upshape(\arabic*)},leftmargin=*]
  %\item Every star-Lindel\"{o}f \cite{sCP} (respectively, strongly star-Lindel\"{o}f \cite{sCP}) space of cardinality less than $\mathfrak{d}$ is star-Scheepers (respectively, strongly star-Scheepers).
 % \item The Isbell-Mr\'owka space $\Psi(\mathcal{A})$ \cite{Mrowka} (respectively, Niemytzki plane $N(Y)$ \cite{Niemytzki}) is strongly star-Scheepers if and only if $|\mathcal{A}|<\mathfrak{d}$ (respectively, $|Y|<\mathfrak{d}$).
  \item If each finite power of a space $X$ is star-Menger (respectively, strongly star-Menger, $\NSUf(\mathcal{O},\mathcal{O})$), then $X$ has the star-Scheepers (respectively, strongly star-Scheepers, $\NSUf(\mathcal{O},\Omega)$) property.
  \item A space $X$ satisfies $\NSUf(\mathcal{O},\Omega)$ if and only if $X$ satisfies $\NSUf(\mathcal{O},\mathcal{O}^{wgp})$.
  \item Every star-Alster (respectively, strongly star-Alster) space is star-Scheepers (respectively, strongly star-Scheepers).
 \item If $|\mathcal{A}|<\aleph_\omega$ (respectively, $|Y|<\aleph_\omega$), then $\Psi(\mathcal{A})$ (respectively, $N(Y)$) is star-Scheepers if and only if $\Psi(\mathcal{A})$ (respectively, $N(Y)$) is strongly star-Scheepers.
  \item If $X$ is a Lindel\"{o}f (respectively, star-Lindel\"{o}f) space which is union of less than $\mathfrak{d}$ star-Hurewicz \cite{sHR} (respectively, Hurewicz) subspaces, then $X$ is star-Scheepers.
  \item If $X$ is a strongly star-Lindel\"{o}f space which is union of less than $\mathfrak{d}$ Hurewicz subspaces, then $X$ is strongly star-Scheepers.
  \item If $X$ is a Scheepers space, then $AD(X)$ is strongly star-Scheepers (and hence star-Scheepers).
  %\item For a space $X$, if $AD(X)$ is star-Scheepers (respectively, strongly star-Scheepers, $\NSUf(\mathcal{O},\Omega)$), then $X$ is star-Scheepers (respectively, strongly star-Scheepers, $\NSUf(\mathcal{O},\Omega)$).
  \item If $X$ is a star-Lindel\"{o}f regular $P$-space and $Y$ is an infinite closed and discrete subset of $X$, then $|Y|<\cof(\Fin(w(X))^\mathbb{N})$. Similar result holds if star-Lindel\"{o}f condition is replaced by the selection hypothesis $\NSUf(\mathcal{O},\Omega)$.
  \item The games $\Guf(\mathcal{O},\Omega)$ and $\SGuf(\mathcal{O},\Omega)$ are equivalent in paracompact Hausdorff spaces.
  \item The games $\Guf(\mathcal{O},\Omega)$ and $\SSGf(\mathcal{O},\Omega)$ are equivalent in metacompact spaces.
\end{enumerate}

The paper is organized as follows. In Section 3, we introduce certain star variations of the Scheepers property. As in the case of Lindel\"{o}f, Menger and Scheepers properties, we show that under similar cardinal assumptions the star and strongly star versions of these properties behave identically. Few interesting  observations in the context of Isbell-Mr\'{o}wka space and Niemytzki plane are presented. In particular, we give a combinatorial characterization of the $\Psi$-space having the star-Scheepers property. Later in this section, we deal with Alexandroff duplicates and further investigate preservation like properties. In Section 4, we introduce another variation of the star-Scheepers property $\NSUf(\mathcal{O},\Omega)$, which we call the new star-Scheepers property. The relationship among the star selection principles and the new star selection principles are outlined into an implication diagram (Figure \ref{dig2}). We also present several observations related to local countable cellularity (introduced here) and monotonically normality. In Section 5, we devote our attention to study games corresponding to the Scheepers property and its star variations. We give another implication diagram (Figure~\ref{dig3}) to summarize the relationship between the winning strategies in the games considered here. In the final section, some open problems are posted.

\section{Preliminaries}
Throughout the paper $(X,\tau)$ stands for a topological space. For undefined notions and terminologies see \cite{Engelking}.

Let $\mathcal{A}$ and $\mathcal{B}$ be collections of open covers of a space $X$. Following \cite{coc1,coc2}, we define

\noindent $\S1(\mathcal{A},\mathcal{B})$: For each sequence $(\mathcal{U}_n)$ of elements of $\mathcal{A}$ there exists a sequence $(V_n)$ such that for each $n$ $V_n\in\mathcal{U}_n$ and $\{V_n : n\in\mathbb{N}\}\in\mathcal{B}$.\\

\noindent $\Sf(\mathcal{A},\mathcal{B})$: For each sequence $(\mathcal{U}_n)$ of elements of $\mathcal{A}$ there exists a sequence $(\mathcal{V}_n)$ such that for each $n$ $\mathcal{V}_n$ is a finite subset of $\mathcal{U}_n$ and $\cup_{n\in\mathbb{N}}\mathcal{V}_n\in\mathcal{B}$.\\

\noindent $\Uf(\mathcal{A},\mathcal{B})$: For each sequence $(\mathcal{U}_n)$ of elements of $\mathcal{A}$ there exists a sequence $(\mathcal{V}_n)$ such that for each $n$ $\mathcal{V}_n$ is a finite subset of $\mathcal{U}_n$ and $\{\cup\mathcal{V}_n : n\in\mathbb{N}\}\in\mathcal{B}$ or $\cup\mathcal{V}_n=X$ for some $n$.

For a subset $A$ of a space $X$ and a collection $\mathcal{P}$ of subsets of $X$, $St(A,\mathcal{P})$ denotes the star of $A$ with respect to $\mathcal{P}$, that is the set $\cup\{B\in\mathcal{P} : A\cap B\neq\emptyset\}$. For $A=\{x\}$, $x\in X$, we write $St(x,\mathcal{P})$ instead of $St(\{x\},\mathcal{P})$ \cite{Engelking}.

In \cite{LjSM}, Ko\v{c}inac introduced star selection principles in the following way.

\noindent $\SS1(\mathcal{A},\mathcal{B})$: For each sequence $(\mathcal{U}_n)$ of elements of $\mathcal{A}$ there exists a sequence $(V_n)$ such that for each $n$ $V_n\in\mathcal{U}_n$ and $\{St(V_n,\mathcal{U}_n) : n\in\mathbb{N}\}\in\mathcal{B}$.\\

\noindent $\SSf(\mathcal{A},\mathcal{B})$: For each sequence $(\mathcal{U}_n)$ of elements of $\mathcal{A}$ there exists a sequence $(\mathcal{V}_n)$ such that for each $n$ $\mathcal{V}_n$ is a finite subset of $\mathcal{U}_n$ and $\cup_{n\in\mathbb{N}}\{St(V,\mathcal{U}_n) : V\in\mathcal{V}_n\}\in\mathcal{B}$.\\

\noindent $\SUf(\mathcal{A},\mathcal{B})$: For each sequence $(\mathcal{U}_n)$ of elements of $\mathcal{A}$ there exists a sequence $(\mathcal{V}_n)$ such that for each $n$ $\mathcal{V}_n$ is a finite subset of $\mathcal{U}_n$ and $\{St(\cup\mathcal{V}_n,\mathcal{U}_n) : n\in\mathbb{N}\}\in\mathcal{B}$ or $St(\cup\mathcal{V}_n,\mathcal{U}_n)=X$ for some $n$.\\

\noindent$\SSS1(\mathcal{A},\mathcal{B})$: For each sequence $(\mathcal{U}_n)$ of elements of $\mathcal{A}$ there exists a sequence $(x_n)$ of elements of $X$ such that $\{St(x_n,\mathcal{U}_n) : n\in\mathbb{N}\}\in\mathcal{B}$.\\

\noindent$\SSSf(\mathcal{A},\mathcal{B})$: For each sequence $(\mathcal{U}_n)$ of elements of $\mathcal{A}$ there exists a sequence $(F_n)$ of finite subsets of $X$ such that $\{St(F_n,\mathcal{U}_n) : n\in\mathbb{N}\}\in\mathcal{B}$.

Let $\mathcal O$ denote the collection of all open covers of $X$. An open cover $\mathcal{U}$ of $X$ is said to be a $\gamma$-cover if each element of $X$ does not belong to at most finitely many members of $\mathcal{U}$ \cite{survey} (see also \cite{coc2,coc1}). We use the symbol $\Gamma$ to denote the collection of all $\gamma$-covers of $X$.
%Similarly as $\gamma$-covers we consider a convention for $\omega$-covers to study the star selection principles.
An open cover $\mathcal{U}$ of $X$ is said to be an $\omega$-cover if for each finite subset $F$ of $X$ there is a set $U\in\mathcal{U}$ such that $F\subseteq U$ \cite{coc1,coc2}.
We use the symbol $\Omega$ to denote the collection of all $\omega$-covers of $X$. An open cover $\mathcal{U}$ of $X$ is said to be large \cite{coc1} if for each $x\in X$ the set $\{U\in\mathcal{U} : x\in U\}$ is infinite. The collection of all large covers of $X$ is denoted by $\Lambda$. Note that $\Gamma\subseteq\Omega\subseteq\Lambda\subseteq\mathcal{O}$.
%An open cover $\mathcal{U}$ of $X$ is said to be groupable \cite{cocVII} if it can be expressed as a countable union of finite, pairwise disjoint subfamilies $\mathcal{U}_n$, $n\in\mathbb{N}$, such that each $x\in X$ belongs to $\cup\mathcal{U}_n$ for all but finitely many $n$. The collection of all groupable covers of $X$ is denoted by $\mathcal{O}^{gp}$.
An open cover $\mathcal{U}$ of $X$ is said to be weakly groupable \cite{cocVIII} if it can be expressed as a countable union of finite, pairwise disjoint subfamilies $\mathcal{U}_n$, $n\in\mathbb{N}$, such that for each finite set $F\subseteq X$ we have $F\subseteq\cup\mathcal{U}_n$ for some $n$. The symbol $\mathcal{O}^{wgp}$ denotes the collection of all weakly groupable covers of $X$. Similarly the symbol $\Lambda^{wgp}$ denotes the collection of all weakly groupable large covers of $X$.

A space $X$ is said to have the Menger (respectively, Scheepers, Hurewicz) property if it satisfies the selection hypothesis $\Sf(\mathcal{O},\mathcal{O})$ (respectively, $\Uf(\mathcal{O},\Omega)$, $\Uf(\mathcal{O},\Gamma)$) \cite{coc1,coc2} (see also \cite{FPM,CST}).

A space $X$ is said to have the (1) star-Menger property, (2) strongly star-Menger property, (3) star-Hurewicz property and (4) strongly star-Hurewicz property if it satisfies the selection hypothesis (1) $\SSf(\mathcal{O},\mathcal{O})$, (2) $\SSSf(\mathcal{O},\mathcal{O})$, (3) $\SUf(\mathcal{O},\Gamma)$ and (4) $\SSSf(\mathcal{O},\Gamma)$ respectively \cite{LjSM,sHR} (see also \cite{rsM,rssM}).

A space $X$ is said to be starcompact (respectively, star-Lindel\"{o}f) if for every open cover $\mathcal{U}$ of $X$ there exists a finite (respectively, countable) set $\mathcal{V}\subseteq\mathcal{U}$ such that $St(\cup\mathcal{V},\mathcal{U})=X$. $X$ is said to be  strongly starcompact (respectively, strongly star-Lindel\"{o}f) if for every open cover $\mathcal{U}$ of $X$ there exists a finite (respectively, countable) set $A\subseteq X$ such that $St(A,\mathcal{U})=X$ \cite{LjSM,sCP}.

A subset $A$ of a space $X$ is said to be regular-closed in $X$ if $\Cl(\Int A)=A$.

A natural pre-order $\leq^*$ on the Baire space $\mathbb{N}^\mathbb{N}$ is defined by $f\leq^*g$ if and only if $f(n)\leq g(n)$ for all but finitely many $n$. A subset $A$ of $\mathbb{N}^\mathbb{N}$ is said to be bounded if there is a $g\in\mathbb{N}^\mathbb{N}$ such that $f\leq^*g$ for all $f\in A$. Let  $\mathfrak{b}$ denote the smallest cardinality of an unbounded subset of $\mathbb{N}^\mathbb{N}$. A subset $D$ of $\mathbb{N}^\mathbb{N}$ is said to be dominating if for each $g\in\mathbb{N}^\mathbb{N}$ there exists a $f\in D$ such that $g\leq^* f$. Let $\mathfrak{d}$ be the minimum cardinality of a dominating subset of $\mathbb{N}^\mathbb{N}$ and $\mathfrak{c}$ be the cardinality of the set of reals. The value of $\mathfrak{d}$ does not change if one considers the relation `$\leq$' instead of `$\leq^*$' \cite{KKJE}. It is well known that $\omega_1\leq\mathfrak{b}\leq\mathfrak{d}\leq\mathfrak{c}$. For any cardinal $\kappa$, $\kappa^+$ denotes the smallest cardinal greater than $\kappa$.

Recall that a family $\mathcal{A}\subseteq P(\mathbb{N})$ is said to be an almost disjoint family if each $A\in\mathcal{A}$ is infinite and for any two distinct elements $B,C\in\mathcal{A}$, $|B\cap C|<\omega$. For an almost disjoint family $\mathcal{A}$, let $\Psi(\mathcal{A})=\mathcal{A}\cup\mathbb{N}$ be the Isbell-Mr\'{o}wka space (or, $\Psi$-space) (see \cite{Mrowka}). It is well known that $\Psi(\mathcal{A})$ is pseudocompact if and only if $\mathcal{A}$ is a maximal almost disjoint family. In general, when talking about Isbell-Mr\'{o}wka space we do not require almost disjoint family to be maximal or the space to be pseudocompact.
%We consider the subspace $[\mathbb{N}]^{\infty}$ of $P(\mathbb N)$ that consists of all infinite subsets of $\mathbb N$. Note that $[\mathbb{N}]^{\infty}$ is homeomorphic to the Baire space. Recall that a set $R\in[\mathbb{N}]^\infty$ reaps a family $\mathcal{R}\subseteq[\mathbb{N}]^\infty$ if for each $P\in\mathcal{R}$ both the sets $P\cap R$ and $P\setminus R$ are infinite. Let $\mathfrak{r}$ be the minimal cardinality of a family $\mathcal{R}\subseteq[\mathbb{N}]^\infty$ that no infinite subset $R$ of $\mathbb N$ reaps \cite{Vaughan}.

For a space $X$, $e(X)=\sup\{|Y| : Y\;\text{is a closed and discrete subspace of}\; X\}$ is said to be the extent of $X$.
%
%For any families $\mathcal{U}$ and $\mathcal{V}$ of subsets of $X$ we denote by $\mathcal{U}\wedge\mathcal{V}$ the set $\{U\cap V : U\in\mathcal{U}\;\text{and}\; V\in\mathcal{V}\}$.

\section{Star versions of the Scheepers property}
\subsection{Star-Scheepers and related spaces}
We first introduce the following definition.
\begin{Def}
\label{D1}
A space $X$ is said to have the star-Scheepers (respectively, strongly star-Scheepers) property if it satisfies the selection hypothesis $\SUf(\mathcal{O},\Omega)$ (respectively, $\SSSf(\mathcal{O},\Omega)$).
\end{Def}
Also a space is called star-Scheepers if it has the star-Scheepers property, and similarly for the other.

The following implication diagram (Figure~\ref{dig1}) of the star variations of the Hurewicz, Scheepers and Menger properties can be easily verified.
\begin{figure}[h]
\begin{adjustbox}{max width=\textwidth,max height=\textheight,keepaspectratio,center}
\begin{tikzcd}[column sep=4ex,row sep=6ex,arrows={crossing over}]
%level 4
\SUf(\mathcal{O},\Gamma) \arrow[rr] && \SUf(\mathcal{O},\Omega) \arrow[rr] && \SSf(\mathcal{O},\mathcal{O})
\\
%level 3
%{\sf SKH} \arrow[rr]\arrow[u]&&
%{\sf SKS} \arrow[rr]\arrow[u]&&
%{\sf SKM}\arrow[u]
%\\
%level 2
\SSSf(\mathcal{O},\Gamma) \arrow[rr]\arrow[u]&&
\SSSf(\mathcal{O},\Omega) \arrow[rr]\arrow[u]&&
\SSSf(\mathcal{O},\mathcal{O}) \arrow[u]
\\
%level 1 (from bottom)
\Uf(\mathcal{O},\Gamma) \arrow[rr]\arrow[u]&&
\Uf(\mathcal{O},\Omega)\arrow[rr]\arrow[u]&&
\Sf(\mathcal{O},\mathcal{O})\arrow[u]
\end{tikzcd}
\end{adjustbox}
\caption{Diagram for star variations of the Hurewicz, Scheepers and Menger properties}
\label{dig1}
\end{figure}

We now present few examples to make distinction between the considered spaces.
The space $X=[0,\omega_1)$ with the usual order topology is a Tychonoff countably compact space. Observe that every countably compact space is strongly starcompact and every strongly starcompact space is strongly star-Hurewicz. Also it is known that strongly starcompact and countably compact are equivalent for Hausdorff spaces \cite{sCP}. Thus $X$ is a Tychonoff strongly star-Scheepers space which is not Scheepers (as it is not Lindel\"{o}f).

Let $aD$ be the one point compactification of the discrete space $D$ with cardinality $\mathfrak c$ and consider the subspace $X=(aD\times [0,\mathfrak{c}^+))\cup (D\times\{\mathfrak{c}^+\})$ of the product space $aD\times [0,\mathfrak{c}^+]$. By \cite[Example 2.2]{sKM}, $X$ is Tychonoff and starcompact which is not strongly star-Menger. Thus there is a star-Scheepers space which is not strongly star-Scheepers.

With necessary modifications of \cite[Theorems 2.4, 2.5]{LjSM} it follows that every metacompact strongly star-Scheepers space is Scheepers and also every meta-Lindel\"{o}f strongly star-Scheepers space is Lindel\"{o}f.
Similar to the other classical selective properties, the Scheepers property is also equivalent to any of its star variations in paracompact Hausdorff spaces.
\begin{Prop}
\label{T1}
For a paracompact Hausdorff space $X$ the following assertions are equivalent.
\begin{enumerate}[wide=0pt,label={\upshape(\arabic*)},
leftmargin=*]
\item $X$ is Scheepers.

\item $X$ is strongly star-Scheepers.

\item $X$ is star-Scheepers.
\end{enumerate}
\end{Prop}
%\begin{proof}
%We only give proof of $(3)\Rightarrow (1)$. Let $(\mathcal{U}_n)$ be a sequence of open covers of $X$. By the Stone characterization of paracompactness \cite{Engelking},  $\mathcal{U}_n$ has an open star-refinement, say $\mathcal{H}_n$ for each $n$. Since $X$ is star-Scheepers, there exists a sequence $(\mathcal{W}_n)$ such that for each $n$ $\mathcal{W}_n$ is a finite subset of $\mathcal{H}_n$ and $\{St(\cup\mathcal{W}_n,\mathcal{H}_n) : n\in\mathbb{N}\}$ is an $\omega$-cover of $X$. For each $W\in\mathcal{W}_n$, choose $U_W\in\mathcal{U}_n$ such that $St(W,\mathcal{H}_n)\subseteq U_W$. Clearly for each $n$ $\mathcal{V}_n=\{U_W : W\in\mathcal{W}_n\}$ is a finite subset of $\mathcal{U}_n$ and the sequence $(\mathcal{V}_n)$ witnesses that $X$ is Scheepers.
%\end{proof}

One can find examples of Hausdorff metacompact star-Scheepers spaces which are not strongly star-Scheepers (and hence not Scheepers). Let $\kappa$ be an infinite cardinal and $D=\{d_\alpha : \alpha<\kappa\}$ be the discrete space of cardinality $\kappa$. Let $aD=D\cup\{\infty\}$ be the one point compactification of $D$. In the product space $aD\times(\omega+1)$, replace the local base of the point $(\infty,\omega)$ by the family $\{U\setminus(D\times\{\omega\}) : (\infty,\omega)\in U\;\text{and}\; U\;\text{is an open set in}\; aD\times(\omega+1)\}$. Let $X$ be the space obtained by such replacement. By \cite[Example 3.4]{SVM}, $X$ is Hausdorff metacompact starcompact (i.e. star-Scheepers) but not strongly star-Lindel\"{o}f (and hence not strongly star-Scheepers). Also by \cite[Theorem 8.10]{SFPH}, there is a set  of reals $X$ which is Scheepers but not Hurewicz. With the help of Proposition~\ref{T1} and \cite[Proposition 4.1]{sHR}, we can conclude that there exists a star-Scheepers (respectively, strongly star-Scheepers) space which is not star-Hurewicz (respectively,  strongly star-Hurewicz). By \cite[Theorem 2.1]{FPM}, there is a set of reals $X$ which is Menger but not Scheepers. Thus Proposition~\ref{T1} and \cite[Theorem 2.8]{LjSM} together imply the existence of a star-Menger (respectively, strongly star-Menger) space which is not star-Scheepers (respectively, strongly star-Scheepers).

\begin{Th}[{cf. \cite[Theorem 2.1]{sHR}}]
\label{T13}
If each finite power of a space $X$ is star-Menger, then $X$ has the star-Scheepers property.
\end{Th}

\begin{Th}[{cf. \cite[Theorem 2.2]{sHR}}]
\label{T14}
For a space $X$ the following assertions are equivalent.
\begin{enumerate}[label={\upshape(\arabic*)}, leftmargin=*]
  \item $X$ is star-Scheepers.
  \item $X$ satisfies $\SUf(\mathcal{O},\mathcal{O}^{wgp})$.
\end{enumerate}
\end{Th}

\begin{Th}[{cf. \cite[Theorem 3.1]{sHR}}]
\label{T15}
If each finite power of a space $X$ is strongly star-Menger, then $X$ has the strongly star-Scheepers property.
\end{Th}

\begin{Th}[{cf. \cite[Theorem 3.2]{sHR}}]
\label{T16}
For a space $X$ the following assertions are equivalent.
\begin{enumerate}[label={\upshape(\arabic*)}, leftmargin=*]
  \item $X$ is strongly star-Scheepers.
  \item $X$ satisfies $\SSSf(\mathcal{O},\mathcal{O}^{wgp})$.
\end{enumerate}
\end{Th}

From \cite{WCPSP,Alster} we recall the following definitions of covers.

\begin{enumerate}[leftmargin=*]
\item[$\mathcal{G}$:] The family of all covers $\mathcal{U}$ of the space $X$ for which each element of $\mathcal{U}$ is a $G_\delta$ set.

\item[$\mathcal{G}_K$:] The family consisting of sets $\mathcal{U}$ where $X$ is not in $\mathcal{U}$, each element of $\mathcal{U}$ is a $G_\delta$ set, and for each compact set $C\subseteq X$ there is a $U\in\mathcal{U}$ such that $C\subseteq U$.
\end{enumerate}

 A space $X$ is said to be Alster \cite{Alster} if each member of $\mathcal{G}_K$ has a countable subset that covers $X$.
By \cite[Proposition 2.6]{LFRR}, a space $X$ is Alster if and only if $X$ satisfies $\S1(\mathcal{G}_K,\mathcal{G})$. The star versions of the Alster property may be introduced as follows.
\begin{Def}
A space $X$ is said to be star-Alster (respectively, strongly star-Alster) if $X$ satisfies $\SS1(\mathcal{G}_K,\mathcal{G})$ (respectively, $\SSS1(\mathcal{G}_K,\mathcal{G})$).
\end{Def}
If a space $X$ has the Alster property, then $X$ is Scheepers. A similar observation for the star versions has been discussed in the following result.
%\begin{Th}
%\label{T45}
%Every Alster space is Scheepers.
%\end{Th}
%\begin{proof}
%Let $X$ be a Alster space. We pick a sequence $(\mathcal{U}_n)$ of open covers of $X$ to show that $X$ is Scheepers. We may assume that for each $n$ $\mathcal{U}_n$ is closed for finite unions. Let $\{N_k : k\in\mathbb{N}\}$ be a partition of $\mathbb{N}$ into infinite sets. For each $k$ and each $n\in N_k$, let $\mathcal{W}_n=\{U^k : U\in\mathcal{U}_n\}$. Then for each $k$ $(\mathcal{W}_n : n\in N_k)$ is a sequence of open covers of $X^k$. Fix $k$. Let $\mathcal{U}=\{\cap_{n\in N_k}W_n : W_n\in\mathcal{W}_n\}$. It can be easily observed that $\mathcal{U}\in\mathcal{G}_K$ for $X^k$. Without loss of generality we suppose that $\mathcal{U}=\mathcal{H}_1\times\mathcal{H}_2\times\cdots\times\mathcal{H}_k$, where for each $1\leq i\leq k$ $\mathcal{H}_i\mathcal{G}_K$ for $X$. Thus for each $1\leq i\leq k$ we can obtain a countable set $\mathcal{C}_i\subseteq\mathcal{H}_i$ that covers $X$ and hence we get a countable set $\mathcal{V}=\mathcal{C}_1\times\mathcal{C}_2\times\cdots\times\mathcal{C}_k\subseteq\mathcal{U}$ that covers $X^k$. We choose $\mathcal{V}=\{\cap_{n\in N_k}W_n^m : W_n^m\in\mathcal{W}_n, m\in N_k\}$. Next we take $\mathcal{V}^\prime=\{W_n^n\in\mathcal{W}_n : n\in N_k \}$. Clearly $\cup\mathcal{V}\subseteq\cup\mathcal{V}^\prime$ and so $\mathcal{V}^\prime$ covers $X^k$. For each $n\in N_k$ we define $\mathcal{V}_n=\{U\in\mathcal{U}_n : U^k\in\mathcal{V}^\prime\}$. Then the sequence $(\mathcal{V}_n)$ guarantees for $(\mathcal{U}_n)$ that $X$ is Scheepers.
%\end{proof}

\begin{Th}
\label{T46}
\hfill
\begin{enumerate}[wide=0pt,label={\upshape(\arabic*)},
ref={\theTh(\arabic*)},leftmargin=*]
  \item\label{T4601} Every star-Alster space is star-Scheepers.
  \item\label{T4602} Every strongly star-Alster space is strongly star-Scheepers.
\end{enumerate}

\end{Th}
\begin{proof}
We only provide proof for $(1)$.
To show that $X$ is star-Scheepers we pick a sequence $(\mathcal{U}_n)$ of open covers of $X$. We may assume that for each $n$ $\mathcal{U}_n$ is closed for finite unions. Let $\{N_k : k\in\mathbb{N}\}$ be a partition of $\mathbb{N}$ into infinite sets. For each $k$ and each $n\in N_k$ choose $\mathcal{W}_n=\{U^k : U\in\mathcal{U}_n\}$. Now for each $k$ $(\mathcal{W}_n : n\in N_k)$ is a sequence of open covers of $X^k$. Fix $k$. Let $\mathcal{U}=\{\cap_{n\in N_k}W_n : W_n\in\mathcal{W}_n\}$. Obviously $\mathcal{U}\in\mathcal{G}_K$ for $X^k$. Without loss of generality we suppose that $\mathcal{U}=\mathcal{H}_1\times\mathcal{H}_2\times\cdots\times\mathcal{H}_k$, where for each $1\leq i\leq k$, $\mathcal{H}_i\in\mathcal{G}_K$ for $X$. Applying the star-Alster property of $X$, for each $1\leq i\leq k$, we get a countable set $\mathcal{C}_i\subseteq\mathcal{H}_i$ such that $\{St(V,\mathcal{H}_i) : V\in\mathcal{C}_i\}\in\mathcal{G}$ for $X$ and subsequently we have a countable set $\mathcal{V}=\mathcal{C}_1\times\mathcal{C}_2\times\cdots\times\mathcal{C}_k\subseteq\mathcal{U}$ such that $St(\cup\mathcal{V},\mathcal{U})$ covers $X^k$. Put $\mathcal{V}=\{\cap_{n\in N_k}W_n^m : W_n^m\in\mathcal{W}_n, m\in N_k\}$. Later we choose $\mathcal{V}^\prime=\{W_n^n\in\mathcal{W}_n : n\in N_k \}$. Clearly $\cup\mathcal{V}\subseteq\cup\mathcal{V}^\prime$ and so $St(\cup\mathcal{V}^\prime,\mathcal{U})$ covers $X^k$. For each $n\in N_k$ let $\mathcal{V}_n=\{U\in\mathcal{U}_n : U^k\in\mathcal{V}^\prime\}$. The sequence $(\mathcal{V}_n)$ witnesses for $(\mathcal{U}_n)$ that $X$ is star-Scheepers.
\end{proof}
%\begin{Th}
%\label{T47}
%Every strongly star-Alster space is strongly star-Scheepers.
%\end{Th}
%\begin{proof}
%Let $(\mathcal{U}_n)$ be a sequence of open covers of $X$. We may assume that for each $n$ $\mathcal{U}_n$ is closed for finite unions. Let $\{N_k : k\in\mathbb{N}\}$ be a partition of $\mathbb{N}$ into infinite sets. For each $k$ and each $n\in N_k$, let $\mathcal{W}_n=\{U^k : U\in\mathcal{U}_n\}$. Then for each $k$ $(\mathcal{W}_n : n\in N_k)$ is a sequence of open covers of $X^k$. Fix $k$. Let $\mathcal{U}=\{\cap_{n\in N_k}W_n : W_n\in\mathcal{W}_n\}$. It can be easily observed that $\mathcal{U}\in\mathcal{G}_K$ for $X^k$. Without loss of generality we suppose that $\mathcal{U}=\mathcal{H}_1\times\mathcal{H}_2\times\cdots\times\mathcal{H}_k$, where for each $1\leq i\leq k$ $\mathcal{H}_i\in\mathcal{G}_K$ for $X$. Applying the strongly star-Alster property of $X$, for each $1\leq i\leq k$, we can obtain a countable set $C_i\subseteq X$ such that $\{St(x,\mathcal{H}_i) : x\in C_i\}\in\mathcal{G}_K$ for $X$ and hence we get a countable set $C=C_1\times C_2\times\cdots\times C_k\subseteq X^k$ such that $St(C,\mathcal{U})$ covers $X^k$. Since $C$ is a countable subset of $X^k$, we enumerate it as $C=\{(x_1^{(n)},x_2^{(n)},\ldots,x_k^{(n)}) : n\in N_k\}$. For each $n\in N_k$, we choose $F_n=\{x_i^{(n)} : 1\leq i\leq k\}$. Then the sequence $(F_n)$ witnesses for $(\mathcal{U}_n)$ that $X$ is strongly star-Scheepers.
%\end{proof}

Let $Y$ be a subspace of a space $X$. We say that $Y$ is star-Scheepers (respectively, strongly star-Scheepers) in $X$ if for each sequence $(\mathcal{U}_n)$ of open covers of $X$ there exists a sequence $(\mathcal{V}_n)$ (respectively, $(F_n)$) such that for each $n$ $\mathcal{V}_n$ is a finite subset of $\mathcal{U}_n$ (respectively, $F_n$ is a finite subset of $X$) and for each finite set $F\subseteq Y$ there exists a $n$ such that $F\subseteq St(\cup\mathcal{V}_n,\mathcal{U}_n)$ (respectively, $F\subseteq St(F_n,\mathcal{U}_n)$). It is immediate that if $Y$ is a star-Scheepers (respectively, strongly star-Scheepers) subspace of $X$, then $Y$ is star-Scheepers (respectively, strongly star-Scheepers) in $X$. Also $X$ is star-Scheepers (respectively, strongly star-Scheepers) if and only if $X$ is star-Scheepers (respectively, strongly star-Scheepers) in $X$. If $Y$ is strongly star-Scheepers in $X$, then $Y$ is star-Scheepers in $X$.

For a set $Y\subseteq\mathbb{N}^\mathbb{N}$, $\maxfin(Y)$ is defined as \[\maxfin(Y)=\left\{\max\{f_1,f_2,\dotsc,f_k\} : f_1,f_2,\dotsc,f_k\in Y\;\text{and}\; k\in\mathbb{N}\right\},\] where $\max\{f_1,f_2,\dotsc,f_k\}(n)=\max\{f_1(n),f_2(n),\dotsc,f_k(n)\}$ for all $n\in\mathbb{N}$.

\begin{Ex}
\label{E19}
\emph{If $Y$ is strongly star-Scheepers (respectively, star-Scheepers) in $X$, then $Y$ need not be a strongly star-Scheepers (respectively, star-Scheepers) subspace of $X$.}\\
Assume $\omega_1<\mathfrak{d}$. Let $X=\Psi(\mathcal{A})$ be the Isbell-Mr\'{o}wka space with $|\mathcal{A}|=\omega_1$.
%Observe that $\mathcal{A}$ is strongly star-Scheepers (and hence star-Scheepers) in $X$, but $\mathcal{A}$ is not a star-Scheepers (and hence not a strongly star-Scheepers) subspace of $X$.
We now show that $\mathcal{A}$ is strongly star-Scheepers in $X$. First choose a sequence $(\mathcal{U}_n)$ of open covers of $X$. Without loss of generality assume that $\mathcal{U}_n=\{U_n(A) : A\in\mathcal{A}\}\cup\{\{n\} : n\in\mathbb{N}\setminus\cup_{A\in\mathcal{A}} U_n(A)\}$ for each $n$. We can further assume that to each $A\in\mathcal{A}$ only one neighbourhood $U_n(A)\in\mathcal{U}_n$ is assigned. For each $A\in\mathcal{A}$ define a function $f_A:\mathbb{N}\to\mathbb{N}$ by $f_A(n)=\min\{m\in\mathbb{N} : m\in U_n(A)\}$ for all $n\in\mathbb{N}$. If $Y=\{f_A : A\in\mathcal{A}\}$, then $\maxfin(Y)$ has cardinality less than $\mathfrak{d}$. Thus there exist a $g\in\mathbb{N}^\mathbb{N}$ and a $n_\mathcal{F}\in\mathbb{N}$ for each finite set $\mathcal{F}\subseteq\mathcal{A}$ such that $f_\mathcal{F}(n_\mathcal{F})< g(n_\mathcal{F})$ with $f_\mathcal{F}\in\maxfin(Y)$. We use the convention that if $\mathcal{F}=\{A\}$, $A\in\mathcal{A}$, then we write $f_A$ instead of $f_\mathcal{F}$. For each $n$ let $F_n=\{1,2,\ldots,g(n)\}$. We claim that the sequence $(F_n)$ witnesses for $(\mathcal{U}_n)$ that $\mathcal{A}$ is strongly star-Scheepers in $X$. Let $\mathcal{F}$ be a finite subset of $\mathcal{A}$. Choose a $n_\mathcal{F}\in\mathbb{N}$ such that $U_{n_\mathcal{F}}(A)\cap F_{n_\mathcal{F}}\neq\emptyset$ for all $A\in\mathcal{F}$. It follows that $\mathcal{F}\subseteq St(F_{n_\mathcal{F}},\mathcal{U}_{n_\mathcal{F}})$. Thus $\mathcal{A}$ is strongly star-Scheepers (and hence star-Scheepers) in $X$.

Since $\mathcal{A}$ is a discrete subspace of $X$ with $|\mathcal{A}|=\omega_1$, $\mathcal{A}$ is not star-Scheepers (and hence not strongly star-Scheepers).
\end{Ex}

\begin{Th}
\label{T26}
\mbox{}
\begin{enumerate}[wide=0pt,label={\upshape(\arabic*)},
ref={\theTh(\arabic*)},leftmargin=*]
  \item\label{T2601} If $X$ is star-Lindel\"{o}f, then every subset of $X$ of cardinality less than $\mathfrak{d}$ is star-Scheepers in $X$.
  \item\label{T2602} If $X$ is strongly star-Lindel\"{o}f, then every subset of $X$ of cardinality less than $\mathfrak{d}$ is strongly star-Scheepers in $X$.
\end{enumerate}
\end{Th}
\begin{proof}
Let $X$ be star-Lindel\"{o}f and $Y$ be a subset of $X$ such that $|Y|<\mathfrak{d}$. Choose a sequence $(\mathcal{U}_n)$ of open covers of $X$. Using the hypothesis we can find for each $n$ a countable subset $\mathcal{V}_n=\{V_m^{(n)} : m\in\mathbb{N}\}$ of $\mathcal{U}_n$ such that $X=St(\cup\mathcal{V}_n,\mathcal{U}_n)$. For each $y\in Y$ choose a function $f_y\in\mathbb{N}^\mathbb{N}$ such that $St(y,\mathcal{U}_n)\cap V_{f_y(n)}^{(n)}\neq\emptyset$ for all $n\in\mathbb{N}$. Since the cardinality of $Z=\{f_y : y\in Y\}$ is less than $\mathfrak{d}$, $\maxfin(Z)$ is also of cardinality less than $\mathfrak{d}$. Thus there exist a $g\in\mathbb{N}^\mathbb{N}$ and for each finite set $F\subseteq Y$ a $n_F\in\mathbb{N}$ such that $f_F(n_F)<g(n_F)$ with $f_F\in\maxfin(Z)$. We use the convention that if $F=\{y\}$, $y\in Y$, then we write $f_y$ instead of $f_F$. For each $n$ let $\mathcal{W}_n=\{V_i^{(n)} : i\leq  g(n)\}$. We now show that the sequence $(\mathcal{W}_n)$ witnesses for $(\mathcal{U}_n)$ that $Y$ is star-Scheepers in $X$. Choose a finite subset $F=\{y_1,y_2,\dotsc,y_k\}$ of $Y$. We claim that $F\subseteq St(\cup\mathcal{W}_{n_F},\mathcal{U}_{n_F})$. Let $z\in F$. Since $St(z,\mathcal{U}_n)\cap V_{f_z(n)}^{(n)}\neq\emptyset$ for all $n\in\mathbb{N}$, we obtain a $U_z\in\mathcal{U}_{n_F}$ containing $z$ such that $U_z\cap V_{f_z(n_F)}^{(n_F)}\neq\emptyset$. Since $f_F(n_F)=\max\{f_{y_1}(n_F),f_{y_2}(n_F),\dotsc,f_{y_k}(n_F)\}$, we have $f_z(n_F)\leq f_F(n_F)$. It follows that $V_{f_z(n_F)}^{(n_F)}\in\mathcal{W}_{n_F}$ and hence $z\in St(\cup\mathcal{W}_{n_F},\mathcal{U}_{n_F})$. Thus $F\subseteq St(\cup\mathcal{W}_{n_F},\mathcal{U}_{n_F})$ and this completes the proof.
\end{proof}

\begin{Cor}
\label{C9}
\hfill
\begin{enumerate}[wide=0pt,label={\upshape(\arabic*)},
ref={\theCor(\arabic*)},leftmargin=*]
  \item\label{C901} Every star-Lindel\"{o}f space of cardinality less than $\mathfrak{d}$ is star-Scheepers.
  \item\label{C902} Every strongly star-Lindel\"{o}f space of cardinality less than $\mathfrak{d}$ is strongly star-Scheepers.
\end{enumerate}
\end{Cor}

As a consequence, we obtain the following result of M. Sakai \cite[Proposition 1.7]{SVM}.
\begin{Cor}
\hfill
\begin{enumerate}[wide=0pt,label={\upshape(\arabic*)},leftmargin=*]
  \item Every star-Lindel\"{o}f space of cardinality less than $\mathfrak{d}$ is star-Menger.
  \item Every strongly star-Lindel\"{o}f space of cardinality less than $\mathfrak{d}$ is strongly star-Menger.
\end{enumerate}
\end{Cor}

We also obtain the following equivalent formulations.
\begin{Cor}
For a space $X$ with cardinality less than $\mathfrak{d}$ the following assertions are equivalent.
\begin{enumerate}[label={\upshape(\arabic*)},
leftmargin=*]
\item $X$ is star-Lindel\"{o}f.

\item $X$ is star-Menger.

\item $X$ is  star-Scheepers.
\end{enumerate}
\end{Cor}

\begin{Cor}
For a space $X$ with cardinality less than $\mathfrak{d}$ the following assertions are equivalent.
\begin{enumerate}[label={\upshape(\arabic*)},
leftmargin=*]
\item $X$ is strongly star-Lindel\"{o}f.

\item $X$ is strongly star-Menger.

\item $X$ is strongly star-Scheepers.
\end{enumerate}
\end{Cor}

\subsection{On the Isbell-Mr\'{o}wka space and Niemytzki plane}
Let $N(X)=(X\times\{0\})\cup(\mathbb{R}\times(0,\infty))$ be the Niemytzki plane on a set $X\subseteq\mathbb{R}$. The topology on $N(X)$ is defined as follows. $\mathbb{R}\times(0,\infty)$ has the Euclidean topology and the set $X\times\{0\}$ has the topology generated by all sets of the form $\{(x,0)\}\cup U$, where $x\in X$ and $U$ is an open disc in $\mathbb{R}\times(0,\infty)$ which is tangent to $X\times\{0\}$ at the point $(x,0)$. The topology on $N(X)$ is also called Niemytzki's tangent disk topology. It is to be noted that Niemytzki originally defined $N(\mathbb{R})$ (see \cite{Niemytzki}).
%The Niemytzki plane on a set $X\subseteq\mathbb{R}$, denoted by $N(X)$, has as underlying set $(X\times\{0\})\cup(\mathbb{R}\times(0,\infty))$. The open half-plane $\mathbb{R}\times(0,\infty)$ has the Euclidean topology and the set $X\times\{0\}$ has the topology generated by all sets of the form $\{(x,0)\}\cup U$, where $x\in X$ and $U$ is an open disc in $\mathbb{R}\times(0,\infty)$ which is tangent to $X\times\{0\}$ at the point $(x,0)$. The Niemytzki plane is also called Niemytzki's tangent disk topology. It is important to mention that $X=\mathbb{R}$ is the way it was originally defined by Niemytzki (see \cite{Niemytzki}).

We say that a space $X$ is $\sigma$-starcompact (respectively, $\sigma$-strongly starcompact) if it can be written as the countable union of starcompact (respectively, strongly starcompact) spaces.

%Let $X$ be a space and $Y=\cup_{k\in\mathbb{N}}X_k$ with $X_k\subseteq X_{k+1}$ for all $k\in\mathbb{N}$. Then for each $k$ $X_k$ is star-Scheepers in $X$ if and only if $Y$ is star-Scheepers in $X$.

\begin{Prop}
\label{T28}
Let $Y=Z\cup K$ be a subspace of $X$. If
\begin{enumerate}[wide=0pt,label={\upshape(\arabic*)},leftmargin=*]
  \item $Z$ is star-Scheepers in $X$ and $K$ is $\sigma$-starcompact, then $Y$ is star-Scheepers in $X$.
  \item $Z$ is strongly star-Scheepers in $X$ and $K$ is $\sigma$-strongly starcompact, then $Y$ is strongly star-Scheepers in $X$.
\end{enumerate}
\end{Prop}
\begin{proof}
Let $Z$ be star-Scheepers in $X$ and $K$ be $\sigma$-starcompact. First we show that for any starcompact subset $C$ of $X$, $Z\cup C$ is star-Scheepers in $X$. Let $(\mathcal{U}_n)$ be a sequence of open covers of $X$. For each $n$ choose a finite set $\mathcal{H}_n\subseteq\mathcal{U}_n$ such that $C\subseteq St(\cup\mathcal{H}_n,\mathcal{U}_n)$. Since $Z$ is star-Scheepers in $X$, there is a sequence $(\mathcal{K}_n)$ such that for each $n$ $\mathcal{K}_n$ is a finite subset of $\mathcal{U}_n$ and for each finite set $F\subseteq Z$ there is a $n$ such that $F\subseteq St(\cup\mathcal{K}_n,\mathcal{U}_n)$. For each $n$ choose $\mathcal{V}_n=\mathcal{H}_n\cup\mathcal{K}_n$. Clearly the sequence $(\mathcal{V}_n)$ witnesses for $(\mathcal{U}_n)$ that $Z\cup C$ is star-Scheepers in $X$. Next without loss of generality assume that $K=\cup_{n\in\mathbb{N}} C_n$, where each $C_n$ is a starcompact subset of $X$ with $C_n\subseteq C_{n+1}$ for all $n$. Thus $Z\cup C_n\subseteq Z\cup C_{n+1}$ for all $n$. It follows that $Y$ is star-Scheepers in $X$.

The result similarly follows when $Z$ is strongly star-Scheepers and $K$ is $\sigma$-strongly starcompact.
\end{proof}
%\begin{Th}
%\label{T30}
%Let $Y=Z\cup K$ be a subspace of $X$. If $Z$ is strongly star-Scheepers in $X$ and $K$ is a $\sigma$-strongly starcompact subset of $X$, then $Y$ is strongly star-Scheepers in $X$.
%\end{Th}
%\begin{proof}
%First we show that for any strongly starcompact subset $C$ of $X$, $Z\cup C$ is strongly star-Scheepers in $X$. Let $(\mathcal{U}_n)$ be a sequence of open covers of $X$. Then for each $n$ there exists a finite $A_n\subseteq C$ such that $C\subseteq St(A_n,\mathcal{U}_n)$. Since $Z$ is strongly star-Scheepers in $X$, there exists a sequence $(B_n)$ of finite subsets of $X$ such that for each finite set $F\subseteq Z$ there is a $n$ such that $F\subseteq St(B_n,\mathcal{U}_n)$. Then the sequence $(A_n\cup B_n)$ witnesses for $(\mathcal{U}_n)$ that $Z\cup C$ is strongly star-Scheepers in $X$.
%
%Without loss of generality let $K=\cup_{n\in\mathbb{N}} C_n$, where each $C_n$ is a strongly starcompact subset of $X$ with $C_n\subseteq C_{n+1}$ for all $n$. Thus $Z\cup C_n\subseteq Z\cup C_{n+1}$ for all $n$ and hence by Proposition~\ref{P12} we can easily conclude that $Y$ is strongly star-Scheepers in $X$.
%\end{proof}

\begin{Cor}
Let $X$ be a space of the form $Y\cup Z$. If
\begin{enumerate}[label={\upshape(\arabic*)},
ref={\theCor(\arabic*)}, leftmargin=*]
  \item\label{C101} $Y$ is star-Scheepers in $X$ and $Z$ is $\sigma$-starcompact, then $X$ is star-Scheepers.
  \item\label{C102} $Y$ is strongly star-Scheepers in $X$ and $Z$ is $\sigma$-strongly starcompact, then $X$ is strongly star-Scheepers.
  \item\label{C103} $Y$ is star-Scheepers and $Z$ is $\sigma$-starcompact, then $X$ is star-Scheepers.
  \item\label{C104} $Y$ is strongly star-Scheepers and $Z$ is $\sigma$-strongly starcompact, then $X$ is strongly star-Scheepers.
\end{enumerate}
\end{Cor}

%For the next two results we assume that all spaces are regular.

Next result follows from \cite[Lemma 3.4]{SSSP}.
\begin{Lemma}
\label{C11}
Let $X$ be a regular space of the form $Y\cup Z$ with $Y\cap Z=\emptyset$, where $Y$ is a closed discrete set and $Z$ is a $\sigma$-compact subset of $X$. If $X$ is strongly star-Scheepers, then $|Y|<\mathfrak{d}$.
\end{Lemma}

In line of \cite[Theorem 3.5]{SSSP}, we obtain the following.
\begin{Th}
\label{T31}
Let $X$ be a regular space of the form $Y\cup Z$ with $Y\cap Z=\emptyset$, where $Y$ is a closed discrete set and $Z$ is a $\sigma$-compact subset of $X$. If $X$ is strongly star-Lindel\"{o}f, then $|Y|<\mathfrak{d}$ if and only if $X$ is strongly star-Scheepers.
\end{Th}
\begin{proof}
Assume that $|Y|<\mathfrak{d}$. By Theorem~\ref{T2602}, $Y$ is strongly star-Scheepers in $X$. Again by Corollary~\ref{C102}, $X$ is strongly star-Scheepers since $Z$ is $\sigma$-strongly starcompact. Conversely if $X$ is strongly star-Scheepers, then by Lemma~\ref{C11}, $|Y|<\mathfrak{d}$.
\end{proof}

\begin{Cor}
\label{C1203}
The following assertions hold.
\begin{enumerate}[label={\upshape(\arabic*)},
ref={\theCor(\arabic*)}, leftmargin=*]
  \item \label{C1201} The Isbell-Mr\'owka space $\Psi(\mathcal{A})$ is strongly star-Scheepers if and only if $|\mathcal{A}|<\mathfrak{d}$.
  \item \label{C1202} The Niemytzki plane $N(Y)$ is strongly star-Scheepers if and only if $|Y|<\mathfrak{d}$.
\end{enumerate}
\end{Cor}

\begin{Cor}
Assume $MA+\neg CH$. The following assertions hold.
\begin{enumerate}[label={\upshape(\arabic*)},
ref={\theCor(\arabic*)}, leftmargin=*]
  \item \label{C1301} If $|\mathcal{A}|<\mathfrak{c}$, then $\Psi(\mathcal{A})$ is strongly star-Scheepers.
  \item \label{C1302} If $|Y|<\mathfrak{c}$, then $N(Y)$ is strongly star-Scheepers.
\end{enumerate}
\end{Cor}
\begin{proof}
 Since $MA$ implies $\mathfrak{d}=\mathfrak{c}$, the proof follows from Corollary~\ref{C1203}.
%$(1)$. Since $MA$ implies $\mathfrak{d}=\mathfrak{c}$, by Corollary~\ref{C1201}, $\Psi(\mathcal{A})$ is strongly star-Scheepers.
%
%$(2)$. Since $MA$ implies $\mathfrak{d}=\mathfrak{c}$, by Corollary~\ref{C1202}, $N(X)$ is strongly star-Scheepers.
\end{proof}

By \cite[Proposition 2.12]{SVM}, for every closed discrete subset $Y$ of a normal star-Menger space $X$, we have $|Y|<\mathfrak{d}$. Thus if $X$ is a normal star-Scheepers space and $Y$ is a closed discrete subset of $X$, then $|Y|<\mathfrak{d}$.
%\begin{Prop}[\!{\cite[Proposition 2.12]{SVM}}]
%Let $X$ be a normal star-Menger space. If $Y$ is a closed discrete subset of $X$, then $|Y|<\mathfrak{d}$ holds. Hence we have $e(X)\leq\mathfrak{d}$.
%\end{Prop}
%
%\begin{Cor}
%\label{C14}
%Let $X$ be a normal star-Scheepers space. If $Y$ is a closed discrete subset of $X$, then $|Y|<\mathfrak{d}$ holds. Hence we have $e(X)\leq\mathfrak{d}$.
%\end{Cor}
In combination with \cite[Corollary 3.6(2)]{SSSP} we obtain the following.
\begin{Prop}
If $N(Y)$ is normal, then the following assertions are equivalent.
\begin{enumerate}[wide=0pt,label={\upshape(\arabic*)}]
  \item $N(Y)$ is star-Menger.
  \item $N(Y)$ is strongly star-Menger.
  \item $N(Y)$ is star-Scheepers.
  \item $N(Y)$ is strongly star-Scheepers.
\end{enumerate}
\end{Prop}

%\begin{Prop}
%\label{P13}
%If $X$ is a space of the form $Y\cup Z$, where $Y$ is a closed discrete set and $Z$ is a $\sigma$-compact subset of $X$, then $X$ is Scheepers if and only if $|Y|<\omega_1$.
%\end{Prop}
If $X=Y\cup Z$, where $Y$ is a closed discrete set and $Z$ is a $\sigma$-compact subset of $X$, then $X$ is Scheepers if and only if $|Y|<\omega_1$. Thus we have the following.
\begin{Prop}
\hfill
\begin{enumerate}[label={\upshape(\arabic*)},
ref={\theCor(\arabic*)}, leftmargin=*]
  \item \label{C1501} $\Psi(\mathcal{A})$ is Scheepers if and only if $|\mathcal{A}|<\omega_1$.
  \item \label{C1502} $N(Y)$ is Scheepers if and only if $|Y|<\omega_1$.
\end{enumerate}
\end{Prop}

\begin{Cor}
Assume $ZFC+\mathfrak{d}=\omega_2$.
\begin{enumerate}[label={\upshape(\arabic*)},ref={\theCor(\arabic*)}, leftmargin=*]
  \item \label{C1601} If $|\mathcal{A}|=\omega_1$, then $\Psi(\mathcal{A})$ is strongly star-Scheepers but not Scheepers.
  \item \label{C1602} If $|Y|=\omega_1$, then $N(Y)$ is strongly star-Scheepers but not Scheepers.
\end{enumerate}
\end{Cor}

In \cite{SCPP}, Bonanzinga and Matveev introduced a cardinal $\mathfrak{d}_\kappa$ for an infinite cardinal $\kappa$. This cardinal $\mathfrak{d}_\kappa$ is also studied in \cite{SVM} and denoted by $\cof(\Fin(\kappa)^\mathbb{N})$. Throughout we use the symbol $\cof(\Fin(\kappa)^\mathbb{N})$ instead of $\mathfrak{d}_\kappa$. For an infinite set $X$ let $\Fin(X)$ denote the set of all finite subsets of $X$. The set $\Fin(X)^\mathbb{N}$ of all functions $f:\mathbb{N}\to\Fin(X)$ is partially ordered coordinate-wise: $f\leq g$ if $f(n)\subseteq g(n)$ for all $n\in\mathbb{N}$. The cofinality of $(\Fin(X)^\mathbb{N},\leq)$ is denoted by $\cof(\Fin(X)^\mathbb{N})$.

The following lemma is required for our next observation.
\begin{Lemma}[{cf. \cite[Lemma 7]{SCPP}}]
\label{L1}
\hfill
\begin{enumerate}[label={\upshape(\arabic*)},
ref={\theLemma(\arabic*)}, leftmargin=*]
  \item \label{L101} $\cof(\Fin(\omega)^\mathbb{N})=\mathfrak{d}$.
  \item \label{L102} If $\omega\leq\kappa\leq\mathfrak{c}$, then $\max\{\mathfrak{d},\kappa\}\leq\cof(\Fin(\kappa)^\mathbb{N})\leq\mathfrak{c}$.
  \item \label{L103} If $\omega\leq\kappa<\aleph_\omega$, then $\cof(\Fin(\kappa)^\mathbb{N})=\max\{\mathfrak{d},\kappa\}$.
  \item \label{L104}  $\cof(\Fin(\mathfrak{c})^\mathbb{N})=\mathfrak{c}$.
\end{enumerate}
\end{Lemma}
If $|\mathcal{A}|=\cof(\Fin(|\mathcal{A}|)^\mathbb{N})$, then $\Psi(\mathcal{A})$ is not star-Menger (and hence not star-Scheepers) (see \cite[Proposition 9]{SCPP}). By \cite[Corollary 3.17]{SSSP}, if $|Y|=\cof(\Fin(|Y|)^\mathbb{N})$, then $N(Y)$ is not star-Menger (and hence not star-Scheepers) (see also \cite{SVM}). Thus if $|\mathcal{A}|=\mathfrak{c}$ (respectively, $|Y|=\mathfrak{c}$), then $\Psi(\mathcal{A})$ (respectively, $N(Y)$) is not star-Menger (and hence not star-Scheepers).
%%%%%%%%%%%%
%\begin{Prop}[{\cite[Proposition 9]{SCPP}}]
%\label{C1701}
%If $|\mathcal{A}|=\cof(\Fin(|\mathcal{A}|)^\mathbb{N})$, then $\Psi(\mathcal{A})$ is not star-Menger (and hence not star-Scheepers).
%\end{Prop}
%
%\begin{Cor}
%\label{C1801} If $|\mathcal{A}|=\mathfrak{c}$, then $\Psi(\mathcal{A})$ is not star-Menger (and hence not star-Scheepers).
%\end{Cor}
%
%\begin{Prop}[{\cite[Corollary 3.17]{SSSP}}]
%\label{C1702}
%If $|Y|=\cof(\Fin(|Y|)^\mathbb{N})$, then $N(Y)$ is not star-Menger (and hence not star-Scheepers).
%\end{Prop}
%
%\begin{Cor}
%\label{C1802} If $|Y|=\mathfrak{c}$, then $N(Y)$ is not star-Menger (and hence not star-Scheepers).
%\end{Cor}
%%%%%%%%%%%%%%%%%%%%%%%%%%%%%%%

%\begin{Cor}
%\hfill
%\begin{enumerate}[label={\upshape(\arabic*)},
%ref={\theCor(\arabic*)}, leftmargin=*]
%  \item \label{C1801} If $|\mathcal{A}|=\mathfrak{c}$, then $\Psi(\mathcal{A})$ is not star-Menger (and hence not star-Scheepers).
%  \item \label{C1802} If $|Y|=\mathfrak{c}$, then $N(Y)$ is not star-Menger (and hence not star-Scheepers).
%\end{enumerate}
%\end{Cor}

\begin{Th}
\hfill
\begin{enumerate}[label={\upshape(\arabic*)},
ref={\theTh(\arabic*)}, leftmargin=*]
  \item \label{T3201} If $|\mathcal{A}|<\aleph_\omega$, then $\Psi(\mathcal{A})$ is star-Scheepers if and only if $\Psi(\mathcal{A})$ is strongly star-Scheepers.
  \item \label{T3202} If $|Y|<\aleph_\omega$, then $N(Y)$ is star-Scheepers if and only if $N(Y)$ is strongly star-Scheepers.
\end{enumerate}
\end{Th}
\begin{proof}
$(1)$. Assume that $|\mathcal{A}|<\mathfrak{d}$. Observe that $\Psi(\mathcal{A})$ is strongly star-Scheepers. Next assume that $\mathfrak{d}\leq |\mathcal{A}|<\aleph_\omega$. By Lemma~\ref{L103}, $|\mathcal{A}|=\cof(\Fin(|\mathcal{A}|)^\mathbb{N})$ and hence $\Psi(\mathcal{A})$ is not star-Scheepers. Thus $\Psi(\mathcal{A})$ is not strongly star-Scheepers.

$(2)$. If $|Y|<\mathfrak{d}$, then $N(Y)$ is strongly star-Scheepers. On the other hand, $\mathfrak{d}\leq |Y|<\aleph_\omega$ gives $|Y|=\cof(\Fin(|Y|)^\mathbb{N})$ (see Lemma~\ref{L103}) and hence $N(Y)$ is not star-Scheepers. Clearly $N(Y)$ is not strongly star-Scheepers.
\end{proof}

%\begin{proof}
%$(1)$. We only present proof of the forward implication. Assume that $|\mathcal{A}|<\mathfrak{d}$. By Corollary~\ref{C1201}, $\Psi(\mathcal{A})$ is strongly star-Scheepers. Next assume that $\mathfrak{d}\leq |\mathcal{A}|<\aleph_\omega$. Then by Lemma~\ref{L103}, $|\mathcal{A}|=\cof(\Fin(|\mathcal{A}|)^\mathbb{N})$ and hence $\Psi(\mathcal{A})$ is not star-Scheepers by Proposition~\ref{C1701}. Thus $\Psi(\mathcal{A})$ is not strongly star-Scheepers.
%
%$(2)$. We only present proof for the forward implication. If $|Y|<\mathfrak{d}$, then by Corollary~\ref{C1202}, $N(Y)$ is strongly star-Scheepers. On the other hand, $\mathfrak{d}\leq |Y|<\aleph_\omega$ gives $|Y|=\cof(\Fin(|Y|)^\mathbb{N})$ (see Lemma~\ref{L103}) and hence by Proposition~\ref{C1702}, $N(Y)$ is not star-Scheepers. Thus $N(Y)$ is not strongly star-Scheepers.
%\end{proof}

%\begin{Cor}
%If $|\mathcal{A}|<\aleph_\omega$, then the following assertions are equivalent.
%\begin{enumerate}[label={\upshape(\arabic*)},
%ref={\theCor(\arabic*)}, leftmargin=*]
%  \item \label{C1901} $\Psi(\mathcal{A})$ is star-Scheepers.
%  \item \label{C1903} $\Psi(\mathcal{A})$ is strongly star-Scheepers.
%\end{enumerate}
%\end{Cor}
%
%\begin{Cor}
%If $|Y|<\aleph_\omega$, then the following assertions are equivalent.
%\begin{enumerate}[label={\upshape(\arabic*)},
%ref={\theCor(\arabic*)}, leftmargin=*]
%  \item \label{C201} $N(Y)$ is star-Scheepers.
%  \item \label{C203} $N(Y)$ is strongly star-Scheepers.
%\end{enumerate}
%\end{Cor}

We also obtain the following reformulation of the star-Scheepers property for $\Psi(\mathcal{A})$ spaces. The proof of the next result uses the technique of \cite[Theorem 2.1]{CASC} with necessary adjustments.
\begin{Th}
\label{T21}
The following assertions are equivalent.
\begin{enumerate}[label={\upshape(\arabic*)}, leftmargin=*]
  \item $\Psi(\mathcal{A})$ has the star-Scheepers property.
  \item For each function $A\mapsto f_A$ from $\mathcal{A}$ to $\mathbb{N}^\mathbb{N}$ there are finite sets $\mathcal{F}_1,\mathcal{F}_2,\cdots\subseteq\mathcal{A}$ such that for each finite set $\mathcal{F}\subseteq\mathcal{A}$ there exists a $n$ such that $(A\setminus f_A(n))\cap\bigcup_{B\in\mathcal{F}_n}(B\setminus f_B(n))\neq\emptyset$ for all $A\in\mathcal{F}$.
\end{enumerate}
\end{Th}
\begin{proof}
$(1)\Rightarrow (2)$. Consider the sequence $(\mathcal{U}_n)$ of open covers of $\Psi(\mathcal{A})$, where $\mathcal{U}_n=\{\{A\}\cup(A\setminus f_A(n)) : A\in\mathcal{A}\}\cup\{\{m\} : m\in\mathbb{N}\}$ for each $n$. Apply the star-Scheepers property of $\Psi(\mathcal{A})$ to $(\mathcal{U}_n)$ to obtain a sequence $(\mathcal{V}_n)$ such that for each $n$ $\mathcal{V}_n$ is a finite subset of $\mathcal{U}_n$ and $\{St(\cup\mathcal{V}_n,\mathcal{U}_n) : n\in\mathbb{N}\}$ is an $\omega$-cover of $\Psi(\mathcal{A})$. For each $n$ and each $\{m\}\in\mathcal{V}_n$ we consider the following two situations. For the first case, if there is an element $B\in\mathcal{A}$ for which $m\in B\setminus f_B(n)$, then replace $\{m\}\in\mathcal{V}_n$ by $\{B\}\cup(B\setminus f_B(n))$. Otherwise if there is no such $B$, we then remove $\{m\}$ from $\mathcal{V}_n$. Next for each $n$ define $\mathcal{F}_n=\{A\in\mathcal{A} : \{A\}\cup(A\setminus f_A(n))\in\mathcal{V}_n\}$. Let $\mathcal{F}$ be a finite subset of $\mathcal{A}$. Choose a $n_0$ such that $\mathcal{F}\subseteq St(\cup\mathcal{V}_{n_0},\mathcal{U}_{n_0})$. It follows that $(\{A\}\cup(A\setminus f_A(n_0)))\cap(\cup\mathcal{V}_{n_0})\neq\emptyset$ for all $A\in\mathcal{F}$. Consequently $(A\setminus f_A(n_0))\cap\bigcup_{B\in\mathcal{F}_{n_0}}(B\setminus f_B(n_0))\neq\emptyset$ for all $A\in\mathcal{F}$.

$(2)\Rightarrow (1)$. Let $(\mathcal{U}_n)$ be a sequence of open covers of $\Psi(\mathcal{A})$. We may assume that for each $A\in\mathcal{A}$ and each $n$ there is a $f_A(n)\in\mathbb{N}$ such that $\{A\}\cup(A\setminus f_A(n))\in\mathcal{U}_n$. Let $\mathcal{F}_1,\mathcal{F}_2,\cdots\subseteq\mathcal{A}$ be finite sets as in $(2)$. Later for each $n$ we define a finite subset $\mathcal{V}_n=\{\{A\}\cup(A\setminus f_A(n)) : A\in\mathcal{F}_n\}$ of $\mathcal{U}_n$. We claim that $\{St(\cup\mathcal{V}_n,\mathcal{U}_n) : n\in\mathbb{N}\}$ is an $\omega$-cover of $\Psi(\mathcal{A})$. Let $\mathcal{F}$ be a finite subset of $\Psi(\mathcal{A})$. We only consider the case when $\mathcal{F}\subseteq\mathcal{A}$ and other cases can be observed similarly. Choose a $n_0$ corresponding to $\mathcal{F}$ as in $(2)$. Then $(A\setminus f_A(n_0))\cap\bigcup_{B\in\mathcal{F}_{n_0}}(B\setminus f_B(n_0))\neq\emptyset$ for all $A\in\mathcal{F}$. This gives us $(\{A\}\cup(A\setminus f_A(n_0)))\cap\bigcup_{B\in\mathcal{F}_{n_0}}(\{B\}\cup(B\setminus f_B(n_0)))\neq\emptyset$ and subsequently $A\in St(\cup\mathcal{V}_{n_0},\mathcal{U}_{n_0})$. Thus $\mathcal{F}\subseteq St(\cup\mathcal{V}_{n_0},\mathcal{U}_{n_0})$ and the proof is now complete.
\end{proof}

\subsection{Preservation under topological operations}
We now study the characteristics of the star Scheepers and related properties under certain topological operations. Observe that the star versions of the Scheepers property are invariants of continuous mappings and are inherited by clopen subsets. In view of \cite[Example 2.2]{sH}, there exists a Tychonoff pseudocompact star-Scheepers space having a regular-closed subset which is not star-Scheepers. We give another counterexample in this context.

\begin{Ex}
\label{E5}
\emph{There exists a Tychonoff strongly star-Scheepers (and hence star-Scheepers) space having a regular-closed $G_\delta$ subset which is not star-Scheepers (and hence not strongly star-Scheepers).}\\
Assume that $\omega_1<\mathfrak{d}$. Let $X=\Psi(\mathcal{A})$ with $|\mathcal{A}|=\omega_1$. By Corollary~\ref{C1201}, $X$ is  Tychonoff and strongly star-Scheepers. Let $D=\{d_\alpha :\alpha<\omega_1\}$ be the discrete space of cardinality $\omega_1$ and let $aD=D\cup\{d\}$ be the one point compactification of $D$. Consider $Y=(aD\times[0,\omega_1])\setminus\{(d,\omega_1)\}$ as a subspace of $aD\times[0,\omega_1]$. To show that $Y$ is not star-Scheepers, it is enough to show that $Y$ is not star-Lindel\"{o}f. For each $\alpha<\omega_1$ let $U_\alpha=\{d_\alpha\}\times(\alpha,\omega_1]$ and $V_\alpha=aD\times[0,\alpha)$. Clearly $U_\alpha\cap U_{\beta}=\emptyset$ for $\alpha\neq\beta$ and $U_\alpha\cap V_{\beta}=\emptyset$ for $\alpha>\beta$. Choose an open cover $\mathcal{U}$ of $Y$, where $\mathcal{U}=\{U_\alpha : \alpha<\omega_1\}\cup \{V_\alpha : \alpha<\omega_1\}$. Also choose a countable subset $\mathcal{V}$ of $\mathcal{U}$ such that $St(\cup\mathcal{V},\mathcal{U})=Y$. We can find $\alpha_0$ and $\beta_0<\omega_1$ such that $U_\alpha\notin\mathcal{V}$ for all $\alpha>\alpha_0$ and $V_\beta\notin\mathcal{V}$ for all $\beta>\beta_0$. Choose a $\gamma<\omega_1$ such that $\gamma>\max\{\alpha_0,\beta_0\}$. As a result $U_\gamma\cap(\cup\mathcal{V})=\emptyset$. Since $U_\gamma$ is the only member of $\mathcal{U}$ containing $(d_\gamma,\omega_1)$, we obtain $(d_\gamma,\omega_1)\notin St(\cup\mathcal{V},\mathcal{U})$. Thus $Y$ is not star-Lindel\"{o}f (and hence not star-Scheepers).

Next assume that $X\cap Y=\emptyset$. Let $f:\mathcal{A}\to D\times\{\omega_1\}$ be a bijection and $Z$ be the quotient image of the topological sum $X\oplus Y$ obtained by identifying $A$ of $X$ with $f(A)$ of $Y$ for every $A\in\mathcal{A}$. Let $q:X\oplus Y\to Z$ be the quotient map. It is immediate that $q(Y)$ is a regular-closed $G_\delta$ subset of $Z$. Since $q(Y)$ is homeomorphic to $Y$, $q(Y)$ is not star-Scheepers. Also since $q(X)$ and $q(aD\times[0,\omega_1))$ are homeomorphic to $X$ and $aD\times[0,\omega_1)$ respectively, $q(X)$ is strongly star-Scheepers and $q(aD\times[0,\omega_1))$ is strongly starcompact. By Corollary~\ref{C104}, $Z=q(X)\cup q(aD\times[0,\omega_1))$ is strongly star-Scheepers.
\end{Ex}

Let $Y=\cup_{k\in\mathbb{N}}X_k\subseteq X$ with $X_k\subseteq X_{k+1}$ for all $k\in\mathbb{N}$. Observe that for each $k$ $X_k$ is
  star-Scheepers (respectively, strongly star-Scheepers) in $X$ if and only if $Y$ is star-Scheepers (respectively, strongly star-Scheepers) in $X$. Next suppose that $X=\cup_{k\in\mathbb{N}}X_k$ with $X_k\subseteq X_{k+1}$ for all $k\in\mathbb{N}$. If each $X_k$ is
 star-Scheepers (respectively, strongly star-Scheepers), then $X$ is also star-Scheepers (respectively, strongly star-Scheepers).
The converse of this assertion is not true.  For example, consider $X=\Psi(\mathcal{A})$ with $|\mathcal{A}|=\omega_1$, under the assumption that $\omega_1<\mathfrak{d}$. By Corollary~\ref{C1201}, $X$ is strongly star-Scheepers (and hence $X$ is star-Scheepers). We can write $X=\cup_{n\in\mathbb{N}}X_n$, where $X_n=\mathcal{A}\cup\{1,2,\ldots,n\}$ for each $n$. Since each $X_n$ is discrete and $|X_n|=\omega_1$, it follows that $X_n$ is not star-Scheepers (and so not strongly star-Scheepers) for any $n$.

%\begin{Ex}
%\label{E1}
%Assume $\omega_1<\mathfrak{d}$. Consider $X=\Psi(\mathcal{A})$ with $|\mathcal{A}|=\omega_1$. Then by Corollary~\ref{C1201}, $X$ is strongly star-Scheepers (and hence $X$ is star-Scheepers). For each $n$ if we consider $X_n=\mathcal{A}\cup\{1,2,\ldots,n\}$, then $X=\cup_{n\in\mathbb{N}}X_n$. Observe that each $X_n$ is discrete with $|X_n|=\omega_1$ and consequently each $X_n$ is not star-Scheepers (and so not strongly star-Scheepers).
%\end{Ex}

On another note by \cite[Theorem 2.5]{FPM}, there exist two Scheepers spaces of reals whose union is not Scheepers (see also \cite[Theorem 3.9]{coc2}). As a consequence of Proposition~\ref{T1}, there exist two star-Scheepers (respectively, strongly star-Scheepers) spaces whose union is not star-Scheepers (respectively, strongly star-Scheepers).

\begin{Th}
\label{T35}
\hfill
\begin{enumerate}[wide=0pt,label={\upshape(\arabic*)},leftmargin=*]
  \item Let $X$ be Lindel\"{o}f. If $X$ is a union of less than $\mathfrak{d}$ star-Hurewicz spaces, then $X$ is star-Scheepers.
  %\item Let $X$ be strongly star-Lindel\"{o}f. If $X$ is a union of less than $\mathfrak{d}$ Hurewicz spaces, then $X$ is strongly star-Scheepers.
  \item Let $X$ be star-Lindel\"{o}f. If $X$ is a union of less than $\mathfrak{d}$ Hurewicz spaces, then $X$ is star-Scheepers.
\end{enumerate}
\end{Th}
\begin{proof}
We only present proof for $(1)$. Let $\kappa$ be a cardinal smaller than $\mathfrak{d}$ and $X=\cup_{\alpha<\kappa} X_\alpha$, where each $X_\alpha$ is star-Hurewicz. Choose a sequence $(\mathcal{U}_n)$ of open covers of $X$ and without loss of generality assume that each $\mathcal{U}_n$ is countable, say  $\mathcal{U}_n=\{U_m^{(n)} : m\in\mathbb{N}\}$. For each $\alpha<\kappa$ there exists a sequence $(\mathcal{V}_n^{(\alpha)})$ such that for each $n$ $\mathcal{V}_n^{(\alpha)}$ is a finite subset of $\mathcal{U}_n$ and each $x\in X_\alpha$ belongs to $St(\cup\mathcal{V}_n^{(\alpha)},\mathcal{U}_n)$ for all but finitely many $n$. Next for each $\alpha<\kappa$ define $f_\alpha:\mathbb{N}\to\mathbb{N}$ by $f_\alpha(n)=\min\{m\in\mathbb{N} : \mathcal{V}_n^{(\alpha)}\subseteq\{U_i^{(n)} : i\leq m\}\}$. If we choose $Y=\{f_\alpha : \alpha<\kappa\}$, then the cardinality of $\maxfin(Y)$ is less than $\mathfrak{d}$. Thus there exists a $g\in\mathbb{N}^\mathbb{N}$ such that for each finite subset $A$ of $\kappa$ we have $g\nleq^* f_A$ with $f_A\in\maxfin(Y)$. For each $n$ $\mathcal{V}_n=\{U_i^{(n)} : i\leq g(n)\}$ is a finite subset of $\mathcal{U}_n$. We now show that the sequence $(\mathcal{V}_n)$ witnesses for $(\mathcal{U}_n)$ that $X$ is star-Scheepers. Let $F$ be a finite subset of $X$. Now choose a finite subset $A$ of $\kappa$ such that $F=\cup_{\alpha\in A}F_\alpha$ with $F_\alpha\subseteq X_\alpha$. For each $\alpha\in A$ consider a $n_\alpha\in\mathbb{N}$ such that $F_\alpha\subseteq St(\cup\mathcal{V}_n^{(\alpha)},\mathcal{U}_n)$ for all $n\geq n_\alpha$. Choose $n_0=\max\{n_\alpha : \alpha\in A\}$ and $n_1\in\mathbb{N}$ such that $n_1>n_0$ and $f_A(n_1)<g(n_1)$. Thus for each $\alpha\in A$ we have $F_\alpha\subseteq St(\cup_{i\leq f_\alpha(n_1)}U_i^{(n_1)},\mathcal{U}_{n_1})\subseteq St(\cup_{i\leq f_A(n_1)}U_i^{(n_1)},\mathcal{U}_{n_1})$ and hence $F\subseteq St(\cup_{i\leq g(n_1)}U_i^{(n_1)},\mathcal{U}_{n_1})\subseteq St(\cup\mathcal{V}_{n_1},\mathcal{U}_{n_1})$. Clearly such $X$ is star-Scheepers.
\end{proof}

Similarly we obtain the following.
\begin{Th}
\label{T36}
Let $X$ be strongly star-Lindel\"{o}f. If $X$ is a union of less than $\mathfrak{d}$ Hurewicz spaces, then $X$ is strongly star-Scheepers.
\end{Th}

Recall that the Alexandroff duplicate $AD(X)$ of a space $X$ (see \cite{AD,Engelking}) is defined as follows. $AD(X)=X\times\{0,1\}$; each point of $X\times\{1\}$ is isolated and a basic neighbourhood of $(x,0)\in X\times\{0\}$ is a set of the form $(U\times\{0\})\cup((U\times\{1\})\setminus\{(x,1)\})$, where $U$ is a neighbourhood of $x$ in $X$.
%\begin{Ex}
%\label{E9}
%There exists a Tychonoff strongly star-Scheepers space $X$ such that $AD(X)$ is not strongly star-Scheepers.\\
%Assuming $\omega_1<\mathfrak{d}$, let $X=\Psi(\mathcal{A})$ with $|\mathcal{A}|=\omega_1$. By Corollary~\ref{C1201}, $X$ is a Tychonoff strongly star-Scheepers space. Suppose that $AD(X)$ is strongly star-Scheepers. Clearly $\mathcal{A}\times\{1\}=\{(A,1) : A\in\mathcal{A}\}$ is a clopen subset of $AD(X)$ with cardinality $\omega_1$. Since each point $(A,1)$ is isolated, $\mathcal{A}\times\{1\}$ is a discrete space. Using the fact that strongly star-Scheepers property preserved under clopen subsets, we get $\mathcal{A}\times\{1\}$ is strongly star-Scheepers, which a contradiction. Consequently $AD(X)$ is not strongly star-Scheepers.
%\end{Ex}
%If we consider the same space $X$ of the above example and slightly modify the proof, then it is easy to see that there exists a Tychonoff star-Scheepers space $X$ such that $AD(X)$ is not star-Scheepers.

\begin{Th}
\label{T41} The following assertions are equivalent.\\
\noindent$(1)$ $X$ is Scheepers.\\
\noindent$(2)$ $AD(X)$ is Scheepers.
\end{Th}
\begin{proof}
Let $(\mathcal{U}_n)$ be a sequence of open covers of $AD(X)$, where $X$ is Scheepers. For each $n$ and each $x\in X$ let $W_x^{(n)}=(V_x^{(n)}\times\{0,1\})\setminus\{(x,1)\}$ be an open set in $AD(X)$ containing $(x,0)$ such that there is a $U_x^{(n)}\in\mathcal{U}_n$ with $W_x^{(n)}\subseteq U_x^{(n)}$, where $V_x^{(n)}$ is an open set in $X$ containing $x$. For each $n$ $\mathcal{W}_n=\{V_x^{(n)} : x\in X\}$ is an open cover of $X$. Apply (1) to $(\mathcal{W}_n)$ to obtain a sequence $(F_n)$ of finite subsets of $X$ such that $(\{V_x^{(n)} : x\in F_n\})$ witnesses the Scheepers property of $X$. For each $n$ and each $x\in F_n$ choose a $O_x^{(n)}\in\mathcal{U}_n$ with $(x,1)\in O_x^{(n)}$. Observe that  $\mathcal{V}_n=\{U_x^{(n)} : x\in F_n\}\cup\{O_x^{(n)} : x\in F_n\}$ is a finite subset of $\mathcal{U}_n$ for each $n$. The sequence $(\mathcal{V}_n)$ witnesses that $AD(X)$ is Scheepers.

Conversely choose a sequence $(\mathcal{U}_n)$ of open covers of $X$, say $\mathcal{U}_n=\{U_x^{(n)} : x\in X\}$, where $U_x^{(n)}$ is an open set in $X$ containing $x$ for each $n$. Define $\mathcal{W}_n=\{(U_x^{(n)}\times\{0,1\})\setminus\{(x,1)\} : x\in X\}\cup\{\{(x,1)\} : x\in X\}$ for each $n$. Since $(\mathcal{W}_n)$ is a sequence of open covers of $AD(X)$, there exists a sequence $(\mathcal{H}_n)$ such that for each $n$ $\mathcal{H}_n$ is a finite subset of $\mathcal{W}_n$ and each finite set $F\subseteq AD(X)$ is contained in $\cup\mathcal{H}_n$ for some $n$. Now $(\mathcal{H}_n)$ produces a sequence $(\mathcal{V}_n)$ of finite subsets of $(\mathcal{U}_n)$ that fulfils the criterion.
%Let $(\mathcal{U}_n)$ be a sequence of open covers of $AD(X)$. For each $n$ and each $x\in X$, let $W_x^{(n)}=(V_x^{(n)}\times\{0,1\})\setminus\{(x,1)\}$ be an open set in $AD(X)$ containing $(x,0)$ such that there is a $U_x^{(n)}\in\mathcal{U}_n$ with $W_x^{(n)}\subseteq U_x^{(n)}$, where $V_x^{(n)}$ is an open set in $X$ containing $x$. For each $n$ $\mathcal{W}_n=\{V_x^{(n)} : x\in X\}$ is an open cover of $X$. Apply the Scheepers property of $X$ to $(\mathcal{W}_n)$ to obtain a sequence $(\mathcal{H}_n)$ such that for each $n$ $\mathcal{H}_n$ is a finite subset of $\mathcal{W}_n$ with each $\mathcal{H}_n=\{V_{x^{(n)}_i}^{(n)} : 1\leq i\leq k_n\}$ and $\{\cup\mathcal{H}_n : n\in\mathbb{N}\}$ is an $\omega$-cover of $X$. For each $n$ and each $1\leq i\leq k_n$, choose a $O_{x^{(n)}_i}^{(n)}\in\mathcal{U}_n$ with $(x^{(n)}_i,1)\in O_{x^{(n)}_i}^{(n)}$. It is easy to see that for each $n$  $\mathcal{V}_n=\{U_{x^{(n)}_i}^{(n)} : 1\leq i\leq k_n\}\cup\{O_{x^{(n)}_i}^{(n)} : 1\leq i\leq k_n\}$ is a finite subset of $\mathcal{U}_n$. Then the sequence $(\mathcal{V}_n)$ guarantees for $(\mathcal{U}_n)$ that $AD(X)$ is Scheepers.
\end{proof}

\begin{Cor}
\label{C28}
If $X$ is Scheepers, then $AD(X)$ is strongly star-Scheepers and hence star-Scheepers.
\end{Cor}

Indeed, for many topological properties $P$ the space $AD(X)$ has $P$ if and only if $X$ has $P$. Such properties are, for instance, Hausdorffness, regularity, Tychonoffness, normality, Lindel\"{o}fness, Menger, Hurewicz, Scheepers (Theorem~\ref{T41}), paracompactness and compactness. For the star variations, the above result is one directional. Consider $X=\Psi(\mathcal{A})$ with $|\mathcal{A}|=\omega_1$ under the hypothesis that $\omega_1<\mathfrak{d}$. By Corollary~\ref{C1201}, $X$ is Tychonoff and strongly star-Scheepers and hence star-Scheepers. But $AD(X)$ is not star-Scheepers and hence not strongly star-Scheepers.

\begin{Th}
\label{T37}
\mbox{}
\begin{enumerate}[wide=0pt,label={\upshape(\arabic*)},
ref={\theTh(\arabic*)},leftmargin=*]
  \item\label{T3701} If $AD(X)$ is star-Scheepers, then $X$ is star-Scheepers.
  \item\label{T3702} If $AD(X)$ is strongly star-Scheepers, then $X$ is strongly star-Scheepers.
\end{enumerate}
\end{Th}
\begin{proof}
Consider the case when $AD(X)$ is star-Scheepers.
Let $(\mathcal{U}_n)$ be a sequence of open covers of $X$. For each $n$ $\mathcal{W}_n=\{U\times\{0,1\} : U\in\mathcal{U}_n\}$ is an open cover of $AD(X)$. Apply the star-Scheepers property of $AD(X)$ to $(\mathcal{W}_n)$ to obtain a sequence $(\mathcal{H}_n)$ such that for each $n$ $\mathcal{H}_n$ is a finite subset of $\mathcal{W}_n$ and $\{St(\cup\mathcal{H}_n,\mathcal{W}_n) : n\in\mathbb{N}\}$ is an $\omega$-cover of $AD(X)$. For each $n$ choose $\mathcal{V}_n=\{U\in\mathcal{U}_n : U\times\{0,1\}\in\mathcal{H}_n\}$. Now the sequence $(\mathcal{V}_n)$ witnesses for $(\mathcal{U}_n)$ that $X$ is star-Scheepers.
%We claim that the sequence $(\mathcal{V}_n)$ witnesses for $(\mathcal{U}_n)$ that $X$ is star-Scheepers. Let $F$ be a finite subset of $X$. Then there exists a $n_0\in\mathbb{N}$ such that $F\times\{0\}\subseteq St(\cup\mathcal{H}_{n_0},\mathcal{W}_{n_0})$. It follows that $F\subseteq St(\cup\mathcal{V}_{n_0},\mathcal{U}_{n_0})$. This completes the proof.
\end{proof}
%\begin{Th}
%\label{T38}
%For a space $X$, if $AD(X)$ is strongly star-Scheepers, then $X$ is also strongly star-Scheepers.
%\end{Th}
%\begin{proof}
%Let $(\mathcal{U}_n)$ be a sequence of open covers of $X$. For each $n$ $\mathcal{W}_n=\{U\times\{0,1\} : U\in\mathcal{U}_n\}$ is an open cover of $AD(X)$. Apply the strongly star-Scheepers property of $AD(X)$ to the sequence $(\mathcal{W}_n)$ to obtain a sequence $(A_n)$ of finite subsets of $AD(X)$ such that $\{St(A_n,\mathcal{W}_n) : n\in\mathbb{N}\}$ is an $\omega$-cover of $AD(X)$. Then we can find a sequence $(F_n)$ of finite subsets of $X$ such that for each $n$ $A_n\subseteq F_n\times\{0,1\}$. Claim that the sequence $(F_n)$ guarantees for $(\mathcal{U}_n)$ that $X$ is strongly star-Scheepers. Let $F$ be a finite subset of $X$. Then we get a $n_0\in\mathbb{N}$ such that $F\times\{0\}\subseteq St(A_{n_0},\mathcal{W}_{n_0})$. It is easy to see that $F\subseteq St(F_{n_0},\mathcal{U}_{n_0})$. Hence the result.
%\end{proof}

The following problems were posed by Song in \cite{rsM,rssM}.
\begin{Prob}[{cf. \cite[Remark 2.2]{rsM}}]
\label{Q1}
Does there exist a space $X$ such that $AD(X)$ is star-Menger, but $X$ is not star-Menger?
\end{Prob}

\begin{Prob}[{cf. \cite[Remark 2.10]{rssM}}]
\label{Q2}
Does there exist a space $X$ such that $AD(X)$ is strongly star-Menger, but $X$ is not strongly star-Menger?
\end{Prob}

 In case of star-Scheepers and strongly star-Scheepers property, the answers to the above problems are not affirmative (see Theorem~\ref{T37}). Likewise it can be shown that answers to the above problems are not affirmative.

The following result uses similar techniques as in Theorem~\ref{T41}.
\begin{Th}
\label{T39}
\hfill
\begin{enumerate}[wide=0pt,label={\upshape(\arabic*)},leftmargin=*]
  \item If $X$ is star-Scheepers, then $X\times\{0\}$ is star-Scheepers in $AD(X)$.
  \item If $X$ is strongly star-Scheepers, then $X\times\{0\}$ is strongly star-Scheepers in $AD(X)$.
\end{enumerate}
\end{Th}

We now consider another cardinal characteristic of the star-Scheepers and related spaces. Since strongly starcompactness is equivalent to countably compactness for Hausdorff spaces, the extent of a Hausdorff strongly starcompact space is finite. The extent of a Scheepers space is countable.
%Under the assumption that $\omega_1<\mathfrak{d}$, $\Psi(\mathcal{A})$, where  $|\mathcal{A}|=\omega_1$, is strongly star-Scheepers and hence star-Scheepers. Since $\mathcal{A}$ is a closed and discrete subspace of $\Psi(\mathcal{A})$, $e(\Psi(\mathcal{A}))=\omega_1$.
Consider $X=(\beta D\times[0,\kappa^+))\cup(D\times\{\kappa^+\})$ as a subspace of $\beta D\times[0,\kappa^+)$, where $D=\{d_\alpha : \alpha<\kappa\}$ is the discrete space of infinite cardinality $\kappa$ and $\beta D$ denotes the \v{C}ech-Stone compactification of $D$. By \cite[Lemma 2.3]{sKH}, $X$ is Tychonoff and star-Scheepers. Since $D\times\{\kappa^+\}$ is a discrete closed set in $X$, we have $e(X)\geq\kappa$. Thus the extent of a Tychonoff star-Scheepers space can be arbitrarily large.
Furthermore, \cite[Example 2.4]{rssM} shows that for every infinite cardinal $\kappa$ there exists a $T_1$ strongly star-Scheepers space $X$ such that $e(X)\geq\kappa$.
%\begin{Ex}
%\label{E10}
%For every infinite cardinal $\kappa$, there exists a $T_1$ strongly star-Scheepers space $X$ such that $e(X)\geq\kappa$.\\
%Let $\kappa$ be an infinite cardinal. Let $D=\{d_\alpha : \alpha<\kappa\}$ be a set of cardinality $\kappa$ and $X=[0,\kappa^+)\cup D$. We define a topology on $X$ as follows. $[0,\kappa^+)$ has the usual order topology and consider as an open subspace of $X$; a basic neighbourhood of a point $d_\alpha\in D$ is of the form $O_\beta(d_\alpha)=\{d_\alpha\}\cup(\beta,\kappa^+)$, where $\beta<\kappa^+$. Then $X$ is a $T_1$ strongly star-Scheepers space with $e(X)\geq\kappa$ (see \cite[Example 2.4]{rssM}).
%\end{Ex}
%For a Tychonoff space $X$, $\beta X$ denotes the \v{C}ech-Stone compactification of $X$.
%\begin{Ex}
%\label{E11}
%For every infinite cardinal $\kappa$, there exists a Tychonoff star-Scheepers space $X$ such that $e(X)\geq\kappa$.\\
%Let $D=\{d_\alpha : \alpha<\kappa\}$ be the discrete space of cardinality $\kappa$. Consider $X=(\beta D\times[0,\kappa^+))\cup(D\times\{\kappa^+\})$ as a subspace of $\beta D\times[0,\kappa^+)$. By \cite[Lemma 2.3]{sKH}, $X$ is a Tychonoff star-Scheepers space. It is easy to see that $D\times\{\kappa^+\}$ is a discrete closed set in $X$. Consequently $e(X)\geq\kappa$.
%\end{Ex}
Observe that if $X$ is a $T_1$ space such that $AD(X)$ is star-Lindel\"{o}f, then $e(X)\leq\omega$. Thus if $X$ is a $T_1$ space such that $AD(X)$ is star-Scheepers (or, strongly star-Scheepers), then the extent of $X$ is countable. Since the converse of Theorem~\ref{T37} does not hold in general, we may ask the following.
\begin{Prob}
\label{Pb1}
\hfill
\begin{enumerate}[wide=0pt,label={\upshape(\arabic*)},leftmargin=*]
  \item Let $X$ be star-Scheepers with $e(X)\leq\omega$. Is the space $AD(X)$ star-Scheepers?
  \item Let $X$ be strongly star-Scheepers with $e(X)\leq\omega$. Is the space $AD(X)$ strongly star-Scheepers?
\end{enumerate}
\end{Prob}

Next we turn to consider preimages. If we assume that $\omega_1<\mathfrak{d}$, then $Y=\Psi(\mathcal{A})$ with $|\mathcal{A}|=\omega_1$ is Tychonoff and strongly star-Scheepers but $X=AD(Y)$ is not star-Scheepers. Since the projection $p:X\to Y$ is a closed 2-to-1 continuous mapping, it follows that preimage of a Tychonoff star-Scheepers (respectively, strongly star-Scheepers) space under a closed 2-to-1 continuous mapping need not be star-Scheepers (respectively, strongly star-Scheepers).
%Readers can verify the following result for the star-Scheepers property.

\begin{Th}
\label{T10}
If $f:X\to Y$ is an open perfect mapping from a space $X$ onto a star-Scheepers space $Y$, then $X$ is star-Scheepers.
\end{Th}
\begin{proof}
We only sketch the proof.
Let $(\mathcal{U}_n)$ be a sequence of open covers of $X$ and $y\in Y$. Since $f^{-1}(y)$ is compact, for each $n$ there exists a finite subset $\mathcal{V}_n^y$ of $\mathcal{U}_n$ such that $f^{-1}(y)\subseteq\cup\mathcal{V}_n^y$ and $f^{-1}(y)\cap U\neq\emptyset$ for each $U\in\mathcal{V}_n^y$. Since $f$ is closed, there exists an open set $U_y^{(n)}$ in $Y$ containing $y$ such that $f^{-1}(U_y^{(n)})\subseteq\cup\mathcal{V}_n^y$. Also we can find an open set $V_y^{(n)}$ in $Y$ containing $y$ such that $V_y^{(n)}\subseteq\cap\{f(U) : U\in\mathcal{V}_n^y\}$ and $f^{-1}(V_y^{(n)})\subseteq f^{-1}(U_y^{(n)})$. Consider the sequence $(\mathcal{V}_n)$ of open covers of $Y$, where $\mathcal{V}_n=\{V_y^{(n)} : y\in Y\}$ for each $n$. Apply the star-Scheepers property of $Y$ to $(\mathcal{V}_n)$ to obtain a sequence $(\mathcal{H}_n)$ such that for each $n$ $\mathcal{H}_n$ is a finite subset of $\mathcal{V}_n$ and $\{St(\cup\mathcal{H}_n,\mathcal{V}_n) : n\in\mathbb{N}\}$ is an $\omega$-cover of $Y$. For each $n$ define $\mathcal{H}_n=\{V_{y_i}^{(n)} : 1\leq i\leq k_n\}$ and $\mathcal{W}_n=\cup_{1\leq i\leq k_n}\mathcal{V}_n^{y_i}$. Since for each $1\leq i\leq k_n$, $f^{-1}(V_{y_i}^{(n)})\subseteq\cup\mathcal{V}_n^{y_i}$, we have $f^{-1}(\cup\mathcal{H}_n)\subseteq\cup\mathcal{W}_n$ for each $n$. The sequence $(\mathcal{W}_n)$ now witnesses that $X$ is star-Scheepers.
\end{proof}

%Using Corollary~\ref{C801} the following result can be obtained.
Since star-Scheepers property is invariant under countable increasing unions, we have the following.
\begin{Cor}
\label{T24}
If $X$ is star-Scheepers  and $Y$ is $\sigma$-compact, then $X\times Y$ is star-Scheepers.
\end{Cor}

The above result does not hold for the strongly star-Scheepers variation. Example~\ref{E12} shows the existence of a strongly star-Scheepers space $X$ and a compact space $Y$ such that their product $X\times Y$ is not strongly star-Lindel\"{o}f (and hence not strongly star-Scheepers). Again by \cite[Example 2.16]{rssM}, there exist two countably compact spaces $X$ and $Y$ whose product $X\times Y$ is not star-Lindel\"{o}f. Thus the product of two star-Scheepers (respectively, strongly star-Scheepers) spaces need not be star-Scheepers (respectively, strongly star-Scheepers).

%By \cite[Example 3.3.3]{sCP}, there exist a countably compact (and hence strongly star-Scheepers) space $X$ and a Lindel\"{o}f space $Y$ such that $X\times Y$ is not strongly star-Lindel\"{o}f.

The following example shows that Corollary~\ref{T24} does not hold if $X$ is star-Scheepers and $Y$ is Lindel\"{o}f.
\begin{Ex}
\label{E17}
\emph{There exist a countably compact space $X$ and a Lindel\"{o}f space $Y$ such that $X\times Y$ is not star-Lindel\"{o}f.}\\
The space $X=[0,\omega_1)$ with the usual order topology is countably compact. Now define a topology on $Y=[0,\omega_1]$ as follows. Each point $\alpha<\omega_1$ is isolated and a set $U$ containing $\omega_1$ is open if and only if $Y\setminus U$ is countable. Clearly $Y$ is Lindel\"{o}f.

%We claim that $X\times Y$ is not star-Lindel\"{o}f.
If possible suppose that $X\times Y$ is star-Lindel\"{o}f. For each $\alpha<\omega_1$ let $U_\alpha=[0,\alpha]\times[\alpha,\omega_1]$ and $V_\alpha=(\alpha,\omega_1)\times\{\alpha\}$. Observe that $U_\alpha\cap V_\beta=\emptyset$ for $\alpha,\beta<\omega_1$ and $V_\alpha\cap V_\beta=\emptyset$ for $\alpha\neq\beta$. Now $\mathcal{U}=\{U_\alpha : \alpha<\omega_1\}\cup\{V_\alpha : \alpha<\omega_1\}$ is an open cover of $X\times Y$. Apply the star-Lindel\"{o}f property of $X\times Y$ to find a countable set $\mathcal{V}\subseteq\mathcal{U}$ such that $St(\cup\mathcal{V},\mathcal{U})=X\times Y$. Since $\mathcal{V}$ is a countable subset of $\mathcal{U}$, there exists a $\alpha_0<\omega_1$ such that $V_\alpha\notin\mathcal{V}$ for each $\alpha>\alpha_0$. If we choose a $\beta_0>\alpha_0$, then $(\beta_0+1,\beta_0)\notin St(\cup\mathcal{V},\mathcal{U})$ as $V_{\beta_0}$ is the only member of $\mathcal{U}$ containing the point $(\beta_0+1,\beta_0)$ and $V_{\beta_0}\cap(\cup\mathcal{V})=\emptyset$, contradicting our assumption. Thus $X\times Y$ can not be star-Lindel\"{o}f.
\end{Ex}

Also note that Theorem~\ref{T10} does not hold for strongly star-Scheepers spaces.
\begin{Ex}
\label{E12}
\emph{There exists a strongly star-Scheepers space whose preimage under an open perfect mapping is not strongly star-Lindel\"{o}f (and hence not strongly star-Scheepers).}\\
Assume that $\omega_1<\mathfrak{d}$. The space $X=\Psi(\mathcal{A})$ with $|\mathcal{A}|=\omega_1$ is strongly star-Scheepers. Let $Y$ be the one point compactification of the discrete space $D=\{d_\alpha : \alpha<\omega_1\}$ (of cardinality $\omega_1$). Observe that the projection mapping $p:X\times Y\to X$ is open perfect.

If possible suppose that the preimage $X\times Y$ is strongly star-Lindel\"{o}f. Since $|\mathcal{A}|=\omega_1$, enumerate $\mathcal{A}$ as $\{A_\alpha : \alpha<\omega_1\}$. Consider the open cover \[\mathcal{U}=\{(\{A_\alpha\}\cup A_\alpha)\times (Y\setminus\{d_\alpha\}) : \alpha<\omega_1\}\cup\{X\times\{d_\alpha\} : \alpha<\omega_1\}\cup\{\{n\}\times Y : n\in\mathbb{N}\}\] of $X\times Y$. Since we assume that $X\times Y$ is strongly star-Lindel\"{o}f, there is a countable subset $A$ of $X\times Y$ such that $St(A,\mathcal{U})=X\times Y$. Also since $A$ is countable, there exists a $\alpha_0<\omega_1$ such that $A\cap(X\times\{d_\alpha\})=\emptyset$ for each $\alpha>\alpha_0$. Choose $\beta>\alpha_0$. Thus $(A_\beta,d_\beta)\notin St(A,\mathcal{U})$ as $X\times\{d_\beta\}$ is the only member of $\mathcal{U}$ containing the point $(A_\beta,d_\beta)$, which is a contradiction. Clearly $X\times Y$ is not strongly star-Lindel\"{o}f.
\end{Ex}

A similar characterization for the strongly star-Scheepers property is the following.
\begin{Th}
\label{T12}
If $f$ is an open, closed and finite-to-one continuous mapping from a space $X$ onto a strongly star-Scheepers space $Y$, then $X$ is strongly star-Scheepers.
\end{Th}

\section{Modifying Scheepers property using new star selection principles}
\subsection{The new star-Scheepers property  $\NSUf(\mathcal{O},\Omega)$}
In \cite{NSP}, the authors introduced the idea of new star selection principles in the following way. Let $\mathcal A$ and $\mathcal B$ be collections of open covers of a space $X$.

$\NSU1(\mathcal{A},\mathcal{B})$: For each sequence $(\mathcal{U}_n)$ of elements of $\mathcal{A}$ there exists a sequence $(U_n)$ such that for each $n$ $U_n\in\mathcal{U}_n$ and $\{St(\cup_{m\in\mathbb{N}}U_m,\mathcal{U}_n) : n\in\mathbb{N}\}\in\mathcal{B}$.\\

$\NSUf(\mathcal{A},\mathcal{B})$: For each sequence $(\mathcal{U}_n)$ of elements of $\mathcal{A}$ there exists a sequence $(\mathcal{V}_n)$ such that for each $n$ $\mathcal{V}_n$ is a finite subset of $\mathcal{U}_n$ and $\{St(\cup_{m\in\mathbb{N}}(\cup\mathcal{V}_m),\mathcal{U}_n) : n\in\mathbb{N}\}\in\mathcal{B}$.\\

Similarly one can define the following new star selection principles (corresponding to the strongly star variations).

$\NSSU1(\mathcal{A},\mathcal{B})$: For each sequence $(\mathcal{U}_n)$ of elements of $\mathcal{A}$ there exists a sequence $(x_n)$ of members of $X$ such that $\{St(\cup_{m\in\mathbb{N}}\{x_m\},\mathcal{U}_n) : n\in\mathbb{N}\}\in\mathcal{B}$.\\

$\NSSUf(\mathcal{A},\mathcal{B})$: For each sequence $(\mathcal{U}_n)$ of elements of $\mathcal{A}$ there exists a sequence $(F_n)$ of finite subsets of $X$ such that $\{St(\cup_{m\in\mathbb{N}}F_m,\mathcal{U}_n) : n\in\mathbb{N}\}\in\mathcal{B}$.

Observe that $\NSSU1(\mathcal{A},\mathcal{B})=\NSSUf(\mathcal{A},\mathcal{B})$.
An easy verification yields the following.
\begin{Prop}
\label{P3}
The following assertions are equivalent.
\begin{enumerate}[label={\upshape(\arabic*)},leftmargin=*]
  \item $X$ is strongly star-Lindel\"{o}f.
  \item $X$ satisfies $\NSSU1(\mathcal{O},\Gamma)$.
  \item $X$ satisfies $\NSSU1(\mathcal{O},\Omega)$.
  \item $X$ satisfies $\NSSU1(\mathcal{O},\mathcal{O})$.
\end{enumerate}
\end{Prop}
%\begin{proof}
%$(1)\Rightarrow (2)$. To show that $X$ satisfies $\NSSU1(\mathcal{O},\Gamma)$ we pick a sequence $(\mathcal{U}_n)$ of open covers of $X$. For $\mathcal{U}_1$ we can find a countable subset $A$ of $X$ such that $St(A,\mathcal{U}_1)=X$, say $A=\{x_n : n\in\mathbb{N}\}$. Then the sequence $(x_n)$ guarantees for $(\mathcal{U}_n)$ that $X$ satisfies $\NSSU1(\mathcal{O},\Gamma)$.
%
%All other implications can be easily verified.
%\end{proof}

\begin{figure}[h!]
\begin{adjustbox}{max width=.8\textwidth,max height=.8\textheight,keepaspectratio,center}
\begin{tikzcd}[column sep=6.5ex,row sep=6ex,arrows={crossing over}]
%level 8
&&\NSUf(\mathcal{O},\Gamma)\arrow[rr]&&
\NSUf(\mathcal{O},\Omega)\arrow[rr]&&
\NSUf(\mathcal{O},\mathcal{O})\arrow[dl]
\\
%level 7
&&&&&\text{star-Lindel\"{o}f}&
\\
%level 6
&\NSU1(\mathcal{O},\Gamma)\arrow[rr]\arrow[uur]&&
\NSU1(\mathcal{O},\Omega)\arrow[rr]\arrow[uur]&&
\NSU1(\mathcal{O},\mathcal{O})\arrow[uur]&
\\
%level 5
&&\SUf(\mathcal{O},\Gamma)\arrow[rr]\arrow[uuu]&&
\SUf(\mathcal{O},\Omega)\arrow[rr]\arrow[uuu]&&
\SSf(\mathcal{O},\mathcal{O})\arrow[uuu]
\\
%level 4
\NSSU1(\mathcal{O},\Gamma)\arrow[thick,equal,rr]\arrow[uur]&&
\NSSU1(\mathcal{O},\Omega)\arrow[thick,equal,rr]\arrow[uur]&&
\NSSU1(\mathcal{O},\mathcal{O})\arrow[uur]&&
\\
%level 3
&\SS1(\mathcal{O},\Gamma)\arrow[rr]\arrow[uur]\arrow[uuu]&&
\SS1(\mathcal{O},\Omega)\arrow[rr]\arrow[uur]\arrow[uuu]&&
\SS1(\mathcal{O},\mathcal{O})\arrow[uur]\arrow[uuu]&
\\
%level 2
&\text{strongly star-Lindel\"{o}f}\arrow[thick,equal,uul]&&&&&
\\
%level 1
\SSSf(\mathcal{O},\Gamma)\arrow[rr]\arrow[uuuurr,bend left=33]\arrow[uuu]&&
\SSSf(\mathcal{O},\Omega)\arrow[rr]\arrow[uuuurr,bend left=33]\arrow[uuu]&&
\SSSf(\mathcal{O},\mathcal{O})\arrow[uuuurr,bend left=33]\arrow[uuu]&&
\end{tikzcd}
\end{adjustbox}
\caption{Diagram for new star selection principles}
\label{dig2}
\end{figure}

We call $\NSUf(\mathcal{O},\Omega)$ as the new star-Scheepers property. The relations among the star selection principles and the new star selection principles are delineated into an implication diagram (Figure~\ref{dig2}).

Note that $\Psi(\mathcal{A})$ with $|\mathcal{A}|=\mathfrak{c}$ is not star-Scheepers, but being a separable space (implies strongly star-Lindel\"{o}f) it satisfies $\NSUf(\mathcal{O},\Omega)$.

Observe that $\NSUf(\mathcal{O},\Omega)$ is an invariant of clopen subsets, continuous mappings and countable increasing unions.
By Example~\ref{E5}, there exists a Tychonoff space with the property $\NSUf(\mathcal{O},\Omega)$ having a regular-closed $G_\delta$ subset which does not satisfy $\NSUf(\mathcal{O},\Omega)$. Also note that $\Psi(\mathcal{A})$ satisfies $\NSUf(\mathcal{O},\Omega)$ for any $\mathcal{A}$. If we consider $|\mathcal{A}|=\omega_1$, then $AD(\Psi(\mathcal{A}))$ does not satisfy $\NSUf(\mathcal{O},\Omega)$ as $\mathcal{A}\times\{1\}=\{(A,1) : A\in\mathcal{A}\}$ is a discrete clopen subset of $AD(X)$ with cardinality $\omega_1$.
The following result is similar to Theorem~\ref{T3701}.
\begin{Th}
\label{T52}
For a space $X$ if $AD(X)$ satisfies $\NSUf(\mathcal{O},\Omega)$, then so does $X$.
\end{Th}
%\begin{proof}
%The proof is similar to the proof of Theorem~\ref{T3701} and so is omitted.
%\end{proof}

%\begin{Th}
%\label{T53}
%If $X$ is a $T_1$ space such that $AD(X)$ is star-Lindel\"{o}f, then $e(X)\leq\omega$.
%\end{Th}
%\begin{proof}
%Suppose that $e(X)\geq\omega_1$. Then there exists a closed discrete subset $D$ of $X$ with $|D|\geq\omega_1$. It is clear that $D\times\{1\}$ is a clopen subset of $AD(X)$ and since each point of $D\times\{1\}$ is isolated, it is also discrete. Thus $AD(X)$ is not star-Lindel\"{o}f as the star-Lindel\"{o}f property is preserved under clopen subsets, which is a contradiction. Consequently $e(X)\leq\omega$.
%\end{proof}

%\begin{Cor}
%\label{C19}
%If $X$ is a $T_1$ space such that $AD(X)$ satisfies $\NSUf(\mathcal{O},\Omega)$ , then $e(X)\leq\omega$.
%\end{Cor}

Also observe that if $X$ is a $T_1$ space such that $AD(X)$ satisfies $\NSUf(\mathcal{O},\Omega)$ , then $e(X)\leq\omega$.

 As in case of the star-Scheepers and strongly star-Scheepers property, $\NSUf(\mathcal{O},\Omega)$ is also not an inverse invariant of closed 2-to-1 continuous mappings. Indeed, $Y=\Psi(\mathcal{A})$ with $|\mathcal{A}|=\omega_1$ satisfies $\NSUf(\mathcal{O},\Omega)$. But $X=AD(Y)$ does not satisfy $\NSUf(\mathcal{O},\Omega)$ and the projection mapping $p:X\to Y$ is closed 2-to-1 continuous. However the following result is obtained.
\begin{Th}
\label{T54}
If $f:X\to Y$ is an open perfect mapping from a space $X$ onto a space $Y$ having the property $\NSUf(\mathcal{O},\Omega)$, then $X$ satisfies $\NSUf(\mathcal{O},\Omega)$.
\end{Th}

Since $\NSUf(\mathcal{O},\Omega)$ is preserved under countable increasing unions, we obtain the following.
\begin{Cor}
\label{C20}
If $X$ satisfies $\NSUf(\mathcal{O},\Omega)$ and $Y$ is $\sigma$-compact, then $X\times Y$ satisfies $\NSUf(\mathcal{O},\Omega)$.
\end{Cor}

The property $\NSUf(\mathcal{O},\Omega)$ is not preserved under finite products (see Example~\ref{E17}). However we obtain the following.
\begin{Th}
\label{T55}
If every finite power of $X$ satisfies $\NSUf(\mathcal{O},\mathcal{O})$, then $X$ satisfies $\NSUf(\mathcal{O},\Omega)$.
\end{Th}
\begin{proof}
The proof is similar to \cite[Theorem 2.1]{sHR} and so is omitted.
\end{proof}

As a consequence, we obtain the following result.
\begin{Cor}[{\cite[Theorem 2.19]{NSP}}]
\label{C22}
If every finite power of a space $X$ satisfies $\NSU1(\mathcal{O},\mathcal{O})$, then $X$ satisfies $\NSUf(\mathcal{O},\Omega)$.
\end{Cor}

The proof of the next result follows from \cite[Theorem 2.2]{sHR} with necessary modifications.
\begin{Th}
\label{T56}
For a space $X$ the following assertions are equivalent.
\begin{enumerate}[wide=0pt,label={\upshape(\arabic*)}]
  \item $X$ satisfies $\NSUf(\mathcal{O},\Omega)$.
  \item $X$ satisfies $\NSUf(\mathcal{O},\mathcal{O}^{wgp})$.
\end{enumerate}
\end{Th}

In combination with Theorem~\ref{T55} we obtain the following.
\begin{Cor}
\label{C23}
If every finite power of a space $X$ satisfies $\NSUf(\mathcal{O},\mathcal{O})$, then $X$ satisfies $\NSUf(\mathcal{O},\mathcal{O}^{wgp})$.
\end{Cor}

The symbol $w(X)$ denotes the weight of a space $X$.
\begin{Th}
\label{NT2}
Let $X$ be a star-Lindel\"{o}f regular $P$-space. If $Y$ is an infinite closed and discrete subset of $X$, then $|Y|<\cof(\Fin(w(X))^\mathbb{N})$.
\end{Th}
\begin{proof}
Let $\kappa=\cof(\Fin(w(X))^\mathbb{N})$. We prove this by contrapositive argument. If possible suppose that $\kappa\leq |Y|$. Since $w(X)\leq\kappa$ and $|Y|\leq w(X)$, by our supposition $\kappa=|Y|$. By Lemma~\ref{L101}, $\kappa$ is uncountable. Choose a base $\mathcal{B}$ for $X$ with $|\mathcal{B}|=w(X)$. We may assume that for each $B\in\mathcal{B}$, $|\overline{B}\cap Y|\leq 1$. Let $\{(\mathcal{B}^{(\alpha)}_n) : \alpha<\kappa\}$ be a cofinal subset of $(\Fin(\mathcal{B})^\mathbb{N},\leq)$. It now remains to show that $X$ is not star-Lindel\"{o}f. For each $y\in Y$ choose an open set $U_y$ in $X$ containing $y$ such that $\overline{U_y}\cap Y=\{y\}$. Let $Z=\{y_\alpha : \alpha<\kappa\}$ be a subset of $Y$ where $y_\alpha\in Y\setminus\cup\{\overline{\cup\mathcal{B}^{(\alpha)}_n} : n\in\mathbb{N}\}$ and $y_\alpha\neq y_\beta$ for $\alpha\neq\beta$. Since $X$ is a regular $P$-space, for each $\alpha<\kappa$ there exists a $V_{y_\alpha}\in\mathcal{B}$ such that $y_\alpha\in V_{y_\alpha}\subseteq U_{y_\alpha}$ and $V_{y_\alpha}\cap(\cup\mathcal{B}^{(\alpha)}_n)=\emptyset$ for all $n$. For each $x\in X\setminus Z$ choose $B_x\in\mathcal{B}$ such that $x\in B_x$ and $B_x\cap Z=\emptyset$. Define $\mathcal{U}_n=\{V_{y_\alpha} : \alpha<\kappa\}\cup\{B_x : x\in X\setminus Z\}$. Clearly $\mathcal{U}$ is an open cover of $X$. Next consider a countable subset $\mathcal{V}$ of $\mathcal{U}$ and  enumerate $\mathcal{V}$ as $\{U_n : n\in\mathbb{N}\}$. For each $n$ define $\mathcal{V}_n=\{U_n\}$. Thus there is a $\alpha_0<\kappa$ such that $\mathcal{V}_n\subseteq\mathcal{B}^{(\alpha_0)}_n$ for all $n$. It follows that $V_{y_{\alpha_0}}\cap U_n\subseteq V_{y_{\alpha_0}}\cap(\cup\mathcal{B}^{(\alpha_0)}_n)=\emptyset$ for all $n$ i.e. $V_{y_{\alpha_0}}\cap(\cup\mathcal{V})=\emptyset$. Since $V_{y_{\alpha_0}}$ is the unique member of $\mathcal{U}$ containing $y_{\alpha_0}$, $y_{\alpha_0}\notin St(\cup\mathcal{V},\mathcal{U})$. As a result, $X$ is not star-Lindel\"{o}f.
\end{proof}

\begin{Cor}
\label{NC2}
For a regular $P$-space $X$ the following assertions hold.
\begin{enumerate}[wide=0pt,label={\upshape(\arabic*)},
ref={\theCor(\arabic*)}, leftmargin=*]
  \item\label{NC201}  Let $X$ satisfy $\NSUf(\mathcal{O},\Omega)$. If $Y$ is an infinite closed and discrete subset of $X$, then $|Y|<\cof(\Fin(w(X))^\mathbb{N})$.
  \item\label{NC202} Let $X$ be either star-Lindel\"{o}f or satisfy $\NSUf(\mathcal{O},\Omega)$. If $w(X)=\mathfrak{c}$, then every closed and discrete subset of $X$ has cardinality less than $\mathfrak{c}$.
\end{enumerate}
\end{Cor}

Note that the role of $P$-space is essential in Theorem~\ref{NT2} and Corollary~\ref{NC2}. Consider the (regular) space $X=\Psi(\mathcal{A})$ with $|\mathcal{A}|=\mathfrak{c}$. Now $X$ satisfies $\NSUf(\mathcal{O},\Omega)$ and hence $X$ is star-Lindel\"{o}f. It is well known that $\mathcal{A}$ is a closed and discrete subset of $X$ and $w(X)=\mathfrak{c}$. By Lemma~\ref{L104}, $|\mathcal{A}|\nless\cof(\Fin(w(X))^\mathbb{N})$.
%\begin{Ex}
%\label{NE1}
%Let $X=\Psi(\mathcal{A})$ with $|\mathcal{A}|=\mathfrak{c}$. Then $X$ is a regular $\NSUf(\mathcal{O},\Omega)$ (star-Lindel\"{o}f) space. It is to be noted that $\mathcal{A}$ is a closed and discrete subset of $X$ and $w(X)=\mathfrak{c}$. By Lemma~\ref{L104}, $|\mathcal{A}|\nless\cof(\Fin(w(X))^\mathbb{N})$.
%\end{Ex}

%\subsection{The star-Lindel\"{o}f and $\NSUf(\mathcal{O},\Omega)$ property}
\subsection{Local countable cellularity and the star operations}
Throughout the section all spaces are assumed to be Hausdorff, unless a specific separation axiom is mentioned. A family of pairwise disjoint nonempty open sets in a space is called a cellular family.
\begin{Def}
\label{D8}
A space is said to have local countable cellularity if every cellular family is locally countable.
\end{Def}

Recall that a space $X$ is called collectionwise normal if for every discrete family $\{F_\alpha : \alpha\in\Lambda\}$ of closed sets of $X$ there exists a pairwise disjoint family $\{U_\alpha : \alpha\in\Lambda\}$ of open sets of $X$ such that $F_\alpha\subseteq U_\alpha$ for every $\alpha\in\Lambda$.

We now recollect few definitions from \cite{WCH,RWCH}.
\begin{enumerate}[wide=0pt,label={\upshape(\arabic*)},leftmargin=*]
  \item A subset $S$ of a space $X$ is said to be separated if there exists a collection $\{U_x : x\in S\}$ of disjoint open sets with $x\in U_x$ for every $x\in S$ and the collection $\{U_x : x\in S\}$ is called a separation of $S$.
  \item A space $X$ is said to be $\leq\kappa$-collectionwise Hausdorff if every closed discrete subset of size $\leq\kappa$ can be separated.
  \item A space $X$ is said to be collectionwise Hausdorff if it is $\leq\kappa$-collectionwise Hausdorff for every cardinal $\kappa$.
  \item A subset $S$ of a space $X$ is said to be weakly separated if it has a subset of size $|S|$ that is separated.
  \item A space $X$ is said to be weakly $\kappa$-collectionwise Hausdorff if every closed discrete subset of size $\kappa$ is weakly separated.
  \item A space $X$ is said to be weakly collectionwise Hausdorff if it is weakly $\kappa$-collectionwise Hausdorff for every cardinal $\kappa$.
\end{enumerate}
%\begin{Def}
%\label{D2}
%A subset $S$ of a space $X$ is said to be separated if there exists a collection $\{U_x : x\in S\}$ of disjoint open sets with $x\in U_x$ for every $x\in S$ and the collection $\{U_x : x\in S\}$ is called a separation of $S$.
%\end{Def}
%
%\begin{Def}
%\label{D3}
%A space $X$ is said to be $\leq\kappa$-collectionwise Hausdorff if every closed discrete subset of size $\leq\kappa$ can be separated.
%\end{Def}
%
%\begin{Def}
%\label{D4}
%A space $X$ is said to be collectionwise Hausdorff if it is $\leq\kappa$-collectionwise Hausdorff for every cardinal $\kappa$.
%\end{Def}
%
%\begin{Def}
%\label{D5}
%A subset $S$ of a space $X$ is said to be weakly separated if it has a subset of size $|S|$ that is separated.
%\end{Def}
%
%\begin{Def}
%\label{D6}
%A space $X$ is said to be weakly $\kappa$-collectionwise Hausdorff if every closed discrete subset of size $\kappa$ is weakly separated.
%\end{Def}
%
%\begin{Def}
%\label{D7}
%A space $X$ is said to be weakly collectionwise Hausdorff if it is weakly $\kappa$-collectionwise Hausdorff for every cardinal $\kappa$.
%\end{Def}

Clearly every collectionwise normal space is collectionwise Hausdorff and every collectionwise Hausdorff space is weakly collectionwise Hausdorff.

\begin{Th}
\label{T48}
If $D$ is an uncountable separated closed set in $X$ with a locally countable separation, then $X$ is not star-Lindel\"{o}f and $X$ does not satisfy $\NSUf(\mathcal{O},\Omega)$.
\end{Th}
\begin{proof}
Let $D=\{d_\alpha : \alpha<\omega_1\}$ and $\mathcal{U}=\{U_\alpha : \alpha<\omega_1\}$ be a locally countable separation of $D$, where $d_\alpha\in U_\alpha$ for each $\alpha<\omega_1$. For each $x\in X\setminus D$ let $U_x$ be an open set in $X\setminus D$ containing $x$ such that $U_x$ intersects only countably many members of $\mathcal{U}$. Clearly $\mathcal{W}=\mathcal{U}\cup\{U_x : x\in X\setminus D\}$ is an open cover of $X$.

We now show that there is no such countable set $\mathcal{V}\subseteq\mathcal{W}$ satisfying $X=St(\cup\mathcal{V},\mathcal{W})$. Indeed, suppose that there exists a countable set $\mathcal{V}\subseteq\mathcal{W}$ such that $X=St(\cup\mathcal{V},\mathcal{W})$. Now $\mathcal{V}$ being a countable subset of $\mathcal{W}$, contains only countably many members of $\{U_x : x\in X\setminus D\}$. Clearly $\mathcal{V}$ intersects only countably many members of $\mathcal{U}$. It follows that there exists a $\beta<\omega_1$ such that $U_\beta\cap(\cup\mathcal{V})=\emptyset$. Thus $d_\beta\notin St(\cup\mathcal{V},\mathcal{W})$ as $U_\beta$ is the only member of $\mathcal{W}$ containing $d_\beta$ and we arrive at a contradiction. Therefore $X$ is not star-Lindel\"{o}f.
\end{proof}

\begin{Cor}
\label{C5}
Suppose that $X$ has local countable cellularity. If $X$ is either star-Lindel\"{o}f or satisfies $\NSUf(\mathcal{O},\Omega)$, then any of the following conditions implies that $e(X)\leq\omega$.
\begin{enumerate}[wide=0pt,label={\upshape(\arabic*)},
ref={\theCor(\arabic*)}, leftmargin=*]
  \item\label{C501} $X$ is weakly $\omega_1$-collectionwise Hausdorff.
  \item\label{C502} $X$ is weakly collectionwise Hausdorff.
  \item\label{C503} $X$ is collectionwise Hausdorff.
  \item\label{C504} $X$ is collectionwise normal.
\end{enumerate}
\end{Cor}

\begin{Th}
\label{P4}
Every $T_1$ star-Lindel\"{o}f space is DCCC.
\end{Th}
\begin{proof}
Let $\mathcal{D}$ be a discrete family of nonempty open sets in a $T_1$ star-Lindel\"{o}f space $X$. Choose a member $x(D)\in D$ for each $D\in\mathcal{D}$ and set $A=\{x(D) : D\in\mathcal{D}\}$. Let $\mathcal{W}$ be the collection of all open sets that meet at most one member of $\mathcal{D}$ and that are disjoint from $A$. Since $\mathcal{D}$ is discrete,  $\mathcal{U}=\mathcal{W}\cup\mathcal{D}$ is an open cover of $X$ and every element of $\mathcal{U}$ intersects at most one member of $\mathcal{D}$. Also since $X$ is star-Lindel\"{o}f, there is a countable subset $\mathcal{V}$ of $\mathcal{U}$ such that $\{St(V,\mathcal{U}) : V\in\mathcal{V}\}$ covers $X$. In order to cover $A$ every member of $\mathcal{D}$ must intersect an element of $\mathcal{U}$. Now $A$ is countable as every element of $\mathcal{V}$ intersects at most one member of $\mathcal{D}$. Clearly $\mathcal{D}$ is countable and hence $X$ is DCCC.
\end{proof}

\begin{Cor}
\label{C6}
Every $T_1$  space satisfying $\NSUf(\mathcal{O},\Omega)$ is DCCC.
\end{Cor}
%It is to be noted that every regular star-Lindel\"{o}f space is DCCC.
%\begin{Prop}
%\label{P4}
%Let $X$ be a star-Lindel\"{o}f space. If $X$ has local countable cellularity, then $X$ is DCCC.
%\end{Prop}
%%\begin{proof}
%%Suppose that $X$ is not DCCC. Then there exists a uncountable discrete (and hence disjoint) family $\{U_\alpha : \alpha<\omega_1\}$ of nonempty open sets of $X$. For each $\alpha<\omega_1$, we choose a $x_\alpha\in U_\alpha$. Clearly $\{x_\alpha : \alpha<\omega_1\}$ is a closed separated set of $X$ with a locally countable separation $\{U_\alpha : \alpha<\omega_1\}$. By Theorem~\ref{T48}, $X$ is not star-Lindel\"{o}f, which is a contradiction. Thus $X$ is DCCC.
%%\end{proof}
%
%\begin{Cor}
%\label{C6}
%If $X$ has local countable cellularity and satisfies $\NSUf(\mathcal{O},\Omega)$, then $X$ is DCCC.
%\end{Cor}

\begin{Ex}
\label{E2}
\emph{There exists a Tychonoff DCCC space $X$ which is not star-Lindel\"{o}f. Hence $X$ does not satisfy $\NSUf(\mathcal{O},\Omega)$.}\\
Let $X=([0,\omega_1]\times[0,\omega_1))\cup(D\times\{\omega_1\})$, where $D$ is the collection of all isolated ordinals in $[0,\omega_1)$. Observe that $X$ is Tychonoff and since $[0,\omega_1]\times[0,\omega_1)$ is a countably compact dense subset of $X$, $X$ is pseudocompact and so is DCCC. Consider the open cover $\mathcal{U}=([0,\omega_1]\times[0,\omega_1))\cup\{\{\alpha\}\times [0,\omega_1] : \alpha\in D\}$ of $X$. Suppose that $\mathcal{V}$ is a countable subset of $\mathcal{U}$ satisfying $X=St(\cup\mathcal{V},\mathcal{U})$. Now choose a $\alpha_0$ such that $(\cup\mathcal{V})\cap(\{\alpha_0\}\times [0,\omega_1])=\emptyset$. Since $\{\alpha_0\}\times [0,\omega_1]$ is the only member of $\mathcal{U}$ containing $(\alpha_0,\omega_1)$, it follows that $(\alpha_0,\omega_1)\notin St(\cup\mathcal{V},\mathcal{U})$, which is a contradiction. Thus $X$ is not star-Lindel\"{o}f.
\end{Ex}

Clearly every separable space is strongly star-Lindel\"{o}f (and hence satisfies $\NSUf(\mathcal{O},\Omega)$). On the other way round, the one point Lindel\"{o}fication $D^*$ of the discrete space $D$ of cardinality $\mathfrak{c}$ is Lindel\"{o}f (and hence strongly star-Lindel\"{o}f) but $D^*$ is not separable.
%\begin{Prop}
%\label{P5}
%If a space $X$ has a countable dense subset, then $X$ is strongly star-Lindel\"{o}f.
%\end{Prop}
%%\begin{proof}
%%Let $\mathcal{U}$ be an open cover of $X$ and $Y$ be the countable dense subset of $X$. Clearly $X=St(Y,\mathcal{U})$ and hence $X$ is strongly star-Lindel\"{o}f.
%%\end{proof}
%
%\begin{Cor}
%\label{C7}
%If a space $X$ has a countable dense subset, then $X$ has the property $\NSUf(\mathcal{O},\Omega)$.
%\end{Cor}
%
%The converse of Proposition~\ref{P5} does not hold, the reason is as follows. Let $D$ be the discrete space of cardinality $\mathfrak{c}$ and $D^*$ be the one point Lindel\"{o}fication of $D$. Then $D^*$ is Lindel\"{o}f (and hence strongly star-Lindel\"{o}f) but $D^*$ does not have a countable dense subset.
%\begin{Ex}
%\label{E6}
%Let $D$ be the discrete space of cardinality $\mathfrak{c}$ and $D^*$ be the one point Lindel\"{o}fication of $D$. Then $D^*$ is Lindel\"{o}f (and hence strongly star-Lindel\"{o}f) but $D^*$ does not have a countable dense subset.
%\end{Ex}

%Recall that a space $X$ is said to be perfect if every closed subset of $X$ is $G_\delta$.
\begin{Th}
\label{P6}
Let $X$ be a star-Lindel\"{o}f perfect space. If $X$ has local countable cellularity, then $X$ is CCC.
\end{Th}
\begin{proof}
Suppose that $X$ is not CCC. Choose an uncountable pairwise disjoint family $\mathcal{U}$ of nonempty open sets of $X$. Say $\mathcal{U}=\{U_\alpha : \alpha<\omega_1\}$. For each $\alpha<\omega_1$ choose $d_\alpha\in U_\alpha$ and set $D=\{d_\alpha : \alpha<\omega_1\}$. Now $D$ is an uncountable discrete subset of $X$. Since $X$ is perfect, by \cite[Lemma 3.5]{SRSS}, we may assume that $D$ is a closed subset of $X$. Clearly $D$ is a separated set with a locally countable separation $\{U_\alpha : \alpha<\omega_1\}$. By Theorem~\ref{T48}, $X$ is not star-Lindel\"{o}f. This completes the proof.
\end{proof}

\begin{Cor}
\label{C10}
Let $X$ be a perfect space having local countable cellularity. If $X$ satisfies $\NSUf(\mathcal{O},\Omega)$, then $X$ is CCC.
\end{Cor}

\begin{Prop}
\label{P7}
Let $X$ be a star-Lindel\"{o}f space having local countable cellularity. Any of the following conditions implies that $X$ is CCC.
\begin{enumerate}[wide=0pt,label={\upshape(\arabic*)},
ref={\theProp(\arabic*)}, leftmargin=*]
  \item\label{P701}  $X$ is a union of countably many closed discrete subsets.
  \item\label{P702} $X$ is a semi-stratifiable space.
  \item\label{P703} $X$ is a Moore space.
\end{enumerate}
\end{Prop}
%\begin{proof}
%$(1)$. Since $X$ can be written as the countable union of closed discrete subsets, $X$ is a perfect space. By Proposition~\ref{P6}, $X$ is CCC.
%
%$(2)$. It follows immediately from the fact that every semi-stratifiable space is perfect.
%
%$(3)$. It follows immediately from the fact that every Moore space is semi-stratifiable.
%\end{proof}

\begin{Cor}
\label{C12}
Let $X$ be a space having local countable cellularity. If $X$ satisfies $\NSUf(\mathcal{O},\Omega)$, then any of the following conditions implies that $X$ is CCC.
\begin{enumerate}[wide=0pt,label={\upshape(\arabic*)},
ref={\theCor(\arabic*)}, leftmargin=*]
  \item\label{C121} $X$ is a union of countably many closed discrete subsets.
  \item\label{C122} $X$ is a semi-stratifiable space.
  \item\label{C123} $X$ is a Moore space.
\end{enumerate}
\end{Cor}

Recall that a space $X$ is said to be monotonically normal if it admits an operator $N$ (called the monotone normality operator) that assigns to any point $x\in X$ and any open set $U\ni x$ an open set $N(x,U)$ such that $x\in N(x,U)\subseteq U$ and for any points $x,y\in X$ and open sets $U,V$ such that $x\in U$ and $y\in V$, it follows from $N(x,U)\cap N(y,V)\neq\emptyset$ that $x\in V$ or $y\in U$ \cite{MNS}.

\begin{Lemma}
\label{L2}
If $X$ is a star-Lindel\"{o}f monotonically normal space with local countable cellularity, then $e(X)\leq\omega$.
\end{Lemma}
%\begin{proof}
%Since a monotonically normal space is (hereditary) collectionwise normal, by Corollary~\ref{C504}, $e(X)\leq\omega$.
%\end{proof}
Recall that a space is stratifiable if and only if it is monotonically normal and semi-stratifiable (\cite[Theorem 2.5]{MNS}). Now suppose that $X$ is a star-Lindel\"{o}f stratifiable space with local countable cellularity. By Lemma~\ref{L2}, $e(X)\leq\omega$. Since every semi-stratifiable space $X$ with $e(X)\leq\omega$ is Lindel\"{o}f (\cite{CSSP}), we have the following.
\begin{Th}
\label{T49}
Every star-Lindel\"{o}f stratifiable space $X$ with local countable cellularity is Lindel\"{o}f.
\end{Th}
%\begin{proof}
% By Lemma~\ref{L2}, $e(X)\leq\omega$. Since every semi-stratifiable space $X$ with $e(X)\leq\omega$ is Lindel\"{o}f (see \cite{CSSP}), it follows that $X$ is Lindel\"{o}f.
%\end{proof}

\begin{Cor}
\label{C13}
Let $X$ be a stratifiable space having local countable cellularity. If $X$ satisfies $\NSUf(\mathcal{O},\Omega)$, then $X$ is Lindel\"{o}f.
\end{Cor}

\begin{Cor}
\label{C15}
Let $X$ be a stratifiable space having local countable cellularity. The following assertions are equivalent.
\begin{enumerate}[wide=0pt,label={\upshape(\arabic*)},leftmargin=*]
  \item $X$ is star-Lindel\"{o}f.
  \item $X$ satisfies $\NSUf(\mathcal{O},\Omega)$.
  \item $X$ is strongly star-Lindel\"{o}f.
  \item $X$ is Lindel\"{o}f.
\end{enumerate}
\end{Cor}

\begin{Th}
\label{T50}
Let $Y=\prod_{i=1}^{n}Y_i$, where each $Y_i$ is a scattered monotonically normal space and let $X\subseteq Y$. If $X$ is a star-Lindel\"{o}f space with local countable cellularity, then $e(X)\leq\omega$.
\end{Th}
\begin{proof}
Suppose that $e(X)=\omega_1$. Choose a closed discrete subset $D$ of $X$ with $|D|=\omega_1$. Consider the projection $p_i:Y\to Y_i$ for each $1\leq i\leq n$. Next choose a $k$, $1\leq k\leq n$, such that $|p_k(D)|=\omega_1$. Let $A$ be the collection of isolated points of $p_k(D)$. Since $Y_k$ is scattered, $A$ is dense in $p_k(D)$ and hence $A$ is uncountable. Now $A$ is a separated set in $Y_k$ since $Y_k$ is hereditary collectionwise normal. Let $\{U_x : x\in A\}$ be a separation of $A$. For each $x\in A$ we choose a $y_x\in D$ such that $p_k(y_x)=x$. Consequently $B=\{y_x : x\in A\}$ is a separated set in $X$ with separation $\{p^{-1}(U_x) : x\in A\}$. Also since $B$ is an uncountable closed set in $X$ and $\{p^{-1}(U_x) : x\in A\}$ is locally countable, $X$ is not star-Lindel\"{o}f by Theorem~\ref{T48}.
\end{proof}

\begin{Cor}
\label{C16}
Let $Y=\prod_{i=1}^{n}Y_i$, where each $Y_i$ is a scattered monotonically normal space and let $X$ be a subspace of $Y$ having local countable cellularity. If $X$ satisfies $\NSUf(\mathcal{O},\Omega)$, then $e(X)\leq\omega$.
\end{Cor}

\begin{Cor}
\label{C17}
Let $\alpha$ be an ordinal and $X$ be a subspace of $\alpha^n$  having local countable cellularity. If $X$ is either star-Lindel\"{o}f or satisfies $\NSUf(\mathcal{O},\Omega)$, then $e(X)\leq\omega$.
\end{Cor}

\begin{Th}
\label{T51}
Let $X$ be a star-Lindel\"{o}f subspace of $\omega_1^\Bbb N$. If $X$ has local countable cellularity, then $e(X)\leq\omega$.
\end{Th}
\begin{proof}
Assume the contrary. Suppose that $e(X)=\omega_1$. Choose a closed discrete subset $D$ of $X$ with cardinality $\omega_1$. For each $n\in\mathbb{N}$ consider the $n$th projection mapping $p_n:\omega_1^\mathbb{N}\to\omega_1$. Now it is possible to find a $k\in\mathbb{N}$ such that $|p_k(D)|=\omega_1$. Since otherwise if for each $n$ $p_n(D)$ is countable, then $\prod_{n\in\mathbb{N}}p_n(D)$ is second countable. Also since $D\subseteq\prod_{n\in\mathbb{N}}p_n(D)$ and every second countable space has countable extent, these lead to a contradiction.

Next observe that if $A$ is the collection of isolated points of $p_k(D)$, then $A$ is uncountable. On the other hand, since $\omega_1$ is hereditary collectionwise normal, $A$ is a separated set in $\omega_1$. Let $\{U_x : x\in A\}$ be a separation of $A$. Now choose for each $x\in A$ a $y_x\in D$ such that $p_k(y_x)=x$. Subsequently $B=\{y_x : x\in A\}$ is a separated set in $X$ with a separation $\{p^{-1}(U_x) : x\in A\}$. Since $B$ is an uncountable closed set in $X$ and $\{p^{-1}(U_x) : x\in A\}$ is locally countable, $X$ is not star-Lindel\"{o}f by Theorem~\ref{T48}.
\end{proof}

\begin{Cor}
\label{C18}
Let $X$ be a subspace of $\omega_1^\mathbb{N}$ satisfying $\NSUf(\mathcal{O},\Omega)$. If $X$ has local countable cellularity, then $e(X)\leq\omega$.
\end{Cor}

%\subsection{Few more observations on $\NSUf(\mathcal{O},\Omega)$}

\section{Game theoretic observations}
In \cite{dcna22}, we recently investigated on games of some star selection principles. A similar type of investigation for the star versions of the Scheepers property has been carried out in this section. Few more related observations are also discussed.

%According to \cite{coc1,LjSM} we define infinitely long game corresponding to each selection principle as follows.

Following \cite{coc1,LjSM}, we consider infinitely long games $\mathsf{\bf G}_{\prod}(\mathcal{O},\mathcal{B})$ corresponding to the selection principles ${\textstyle \prod}(\mathcal{O},\mathcal{B})$ and $\mathsf{\bf G}_{\sum}(\mathcal{O},\mathcal{O})$ corresponding to the selection principles ${\textstyle\sum}(\mathcal{O},\mathcal{O})$, where \begin{eqnarray*}
{\textstyle \prod}&\in&\{\NSU1,\SSSf,\Uf,\SUf,\NSUf\},\\
\mathsf{\bf G}_{\prod}&\in&\{\NSG1,\SSGf,\Guf,\SGuf,\NSGuf\},\\
{\textstyle \sum} &\in&\{\S1,\SS1,\SSS1,\NSU1,\Sf,\SSf,\SSSf,\NSUf\},\\
\mathsf{\bf G}_{\sum}&\in&\{\G1,\SG1,\SSG1,\NSG1,\Gf,\SGf,\SSGf,\NSGuf\}
\end{eqnarray*} and $\mathcal{B}\in\{\Omega,\Gamma\}$.

%The game $\G1(\mathcal{O},\mathcal{O})$ (respectively, $\Gf(\mathcal{O},\mathcal{O})$, $\Guf(\mathcal{O},\Gamma)$, $\SG1(\mathcal{O},\mathcal{O})$, $\SGf(\mathcal{O},\mathcal{O})$, $\SGuf(\mathcal{O},\Gamma)$, $\SSG1(\mathcal{O},\mathcal{O})$, $\SSGf(\mathcal{O},\mathcal{O})$, $\SSGf(\mathcal{O},\Omega)$, $\SSGf(\mathcal{O},\Gamma)$, $\NSG1(\mathcal{O},\mathcal{O})$, $\NSG1(\mathcal{O},\Omega)$, $\NSG1(\mathcal{O},\Gamma)$, $\NSGuf(\mathcal{O},\mathcal{O})$, $\NSGuf(\mathcal{O},\Omega)$, $\NSGuf(\mathcal{O},\Gamma)$) on a space $X$ corresponding to the selection principle $\S1(\mathcal{O},\mathcal{O})$ (respectively, $\Sf(\mathcal{O},\mathcal{O})$, $\Uf(\mathcal{O},\Gamma)$, $\SS1(\mathcal{O},\mathcal{O})$, $\SSf(\mathcal{O},\mathcal{O})$, $\SUf(\mathcal{O},\Gamma)$, $\SSS1(\mathcal{O},\mathcal{O})$, $\SSSf(\mathcal{O},\mathcal{O})$, $\SSSf(\mathcal{O},\Omega)$, $\SSSf(\mathcal{O},\Gamma)$, $\NSU1(\mathcal{O},\mathcal{O})$, $\NSU1(\mathcal{O},\Omega)$, $\NSU1(\mathcal{O},\Gamma)$, $\NSUf(\mathcal{O},\mathcal{O})$, $\NSUf(\mathcal{O},\Omega)$, $\NSUf(\mathcal{O},\Gamma)$) can be similarly defined.

The game $\Guf(\mathcal{O},\Omega)$ on a space $X$ corresponding to the selection principle $\Uf(\mathcal{O},\Omega)$ is played as follows. Players ONE and TWO play an inning for each positive integer $n$. In the $n$th inning ONE chooses an open cover $\mathcal{U}_n$ of $X$ and TWO responds by selecting a finite subset $\mathcal{V}_n$ of $\mathcal{U}_n$. TWO wins the play $\mathcal{U}_1, \mathcal{V}_1, \mathcal{U}_2, \mathcal{V}_2,\ldots, \mathcal{U}_n, \mathcal{V}_n,\ldots$ of this game if $\{\cup\mathcal{V}_n : n\in\mathbb{N}\}$ is an $\omega$-cover of $X$; otherwise ONE wins. Other games can be similarly defined.

It is easy to see that if ONE does not have a winning strategy in any of the above game on $X$, then $X$ satisfies the selection principle corresponding to that game.

%The game $\SGuf(\mathcal{O},\Omega)$ on a space $X$ corresponding to the selection principle $\SUf(\mathcal{O},\Omega)$ is also similarly played as $\Guf(\mathcal{O},\Omega)$. Here TWO wins the run of the game if $\{St(\cup\mathcal{V}_n,\mathcal{U}_n) : n\in\mathbb{N}\}$ is an $\omega$-cover of $X$; otherwise ONE wins.

%It is easy to see that if ONE does not have a winning strategy in the game $\G1(\mathcal{O},\mathcal{O})$ (respectively, $\Gf(\mathcal{O},\mathcal{O})$, $\Guf(\mathcal{O},\Omega)$, $\Guf(\mathcal{O},\Gamma)$, $\SG1(\mathcal{O},\mathcal{O})$, $\SGf(\mathcal{O},\mathcal{O})$, $\SGuf(\mathcal{O},\Omega)$, $\SGuf(\mathcal{O},\Gamma)$, $\SSG1(\mathcal{O},\mathcal{O})$, $\SSGf(\mathcal{O},\mathcal{O})$, $\SSGf(\mathcal{O},\Omega)$, $\SSGf(\mathcal{O},\Gamma)$, $\NSG1(\mathcal{O},\mathcal{O})$, $\NSG1(\mathcal{O},\Omega)$, $\NSG1(\mathcal{O},\Gamma)$, $\NSGuf(\mathcal{O},\mathcal{O})$, $\NSGuf(\mathcal{O},\Omega)$, $\NSGuf(\mathcal{O},\Gamma)$) on a space $X$, then $X$ satisfies $\S1(\mathcal{O},\mathcal{O})$ (respectively, $\Sf(\mathcal{O},\mathcal{O})$, $\Uf(\mathcal{O},\Omega)$, $\Uf(\mathcal{O},\Gamma)$, $\SS1(\mathcal{O},\mathcal{O})$, $\SSf(\mathcal{O},\mathcal{O})$, $\SUf(\mathcal{O},\Omega)$, $\SUf(\mathcal{O},\Gamma)$, $\SSS1(\mathcal{O},\mathcal{O})$, $\SSSf(\mathcal{O},\mathcal{O})$, $\SSSf(\mathcal{O},\Omega)$, $\SSSf(\mathcal{O},\Gamma)$, $\NSU1(\mathcal{O},\mathcal{O})$, $\NSU1(\mathcal{O},\Omega)$, $\NSU1(\mathcal{O},\Gamma)$, $\NSUf(\mathcal{O},\mathcal{O})$, $\NSUf(\mathcal{O},\Omega)$, $\NSUf(\mathcal{O},\Gamma)$).

Recall that two games are said to be equivalent if whenever one of the players has a winning strategy in one of the games, then that same player has a winning strategy in the other game \cite{Pearl}.

\begin{Th}
\label{TG11}
For a paracompact Hausdorff space $X$ the games $\Guf(\mathcal{O},\Omega)$ and $\SGuf(\mathcal{O},\Omega)$ are equivalent.
\end{Th}
\begin{proof}
An easy verification shows that if ONE has a winning strategy in $\SGuf(\mathcal{O},\Omega)$ on $X$, then ONE has a winning strategy in $\Guf(\mathcal{O},\Omega)$ on $X$ and on the other hand, if TWO has a winning strategy in $\Guf(\mathcal{O},\Omega)$ on $X$, then TWO has a winning strategy in $\SGuf(\mathcal{O},\Omega)$ on $X$.

Suppose that ONE has a winning strategy $\sigma$ in $\Guf(\mathcal{O},\Omega)$ on $X$. We now define a strategy $\tau$ for ONE in $\SGuf(\mathcal{O},\Omega)$ on $X$ as follows. Let $\sigma(\emptyset)=\mathcal{U}$ be the first move of ONE in $\Guf(\mathcal{O},\Omega)$. Since $X$ is a paracompact Hausdorff space, $\mathcal{U}$ has an open star-refinement, say $\mathcal{W}$. Suppose that the first move of ONE in $\SGuf(\mathcal{O},\Omega)$ is $\tau(\emptyset)=\mathcal{W}$ and TWO responds by selecting a finite subset $\mathcal{W}_1\subseteq\tau(\emptyset)$. For each $W\in\mathcal{W}_1$ we can find a $U_W\in\mathcal{U}$ such that $St(W,\tau(\emptyset))\subseteq U_W$. Thus we obtain a finite subset $\mathcal{V}_1=\{U_W : W\in\mathcal{W}_1\}$ of $\mathcal{U}$. Let $\mathcal{V}_1$ be the response of TWO in $\Guf(\mathcal{O},\Omega)$. The second move of ONE is $\sigma(\mathcal{V}_1)$. Continuing in this way, we obtain the legitimate strategy $\tau$ for ONE in $\SGuf(\mathcal{O},\Omega)$. It now follows that $\tau$ is a winning strategy for ONE in $\SGuf(\mathcal{O},\Omega)$ on $X$.

Next suppose that TWO has a winning strategy $\sigma$ in $\SGuf(\mathcal{O},\Omega)$ on $X$. We define a strategy $\tau$ for TWO in $\Guf(\mathcal{O},\Omega)$ on $X$ as follows. Let $\mathcal{U}_1$ be the first move of ONE in $\Guf(\mathcal{O},\Omega)$. Again by paracompactness and  Hausdorffness of $X$, $\mathcal{U}_1$ has an open star-refinement, say $\mathcal{W}_1$. Define ONE's first move in $\SGuf(\mathcal{O},\Omega)$ to be $\mathcal{W}_1$. TWO responds by choosing a finite subset $\sigma(\mathcal{W}_1)=\mathcal{F}_1$ of $\mathcal{W}_1$ in $\SGuf(\mathcal{O},\Omega)$. Since $\mathcal{W}_1$ is an open star-refinement of $\mathcal{U}_1$, for each $W\in\mathcal{F}_1$ there exists a $U_W\in\mathcal{U}_1$ such that $St(W,\mathcal{W}_1)\subseteq U_W$. Choose $\mathcal{V}_1=\{U_W : W\in\mathcal{F}_1\}$, which is a finite subset of $\mathcal{U}_1$. Define TWO's response in $\Guf(\mathcal{O},\Omega)$ as $\tau(\mathcal{U}_1)=\mathcal{V}_1$. Proceeding similarly, we can construct the strategy $\tau$ for TWO in $\Guf(\mathcal{O},\Omega)$. Observe that $\tau$ is a winning strategy for TWO in $\Guf(\mathcal{O},\Omega)$ on $X$. This completes the proof.
\end{proof}

Similarly the next result is obtained.
%The proof of the next result is similar to the above theorem with necessary modifications and so is omitted.
\begin{Th}
\label{TG13}
For a metacompact space $X$ the games $\Guf(\mathcal{O},\Omega)$ and $\SSGf(\mathcal{O},\Omega)$ are equivalent.
\end{Th}

In association with Theorem~\ref{TG11} and Theorem~\ref{TG13}, we obtain the following.
\begin{Cor}
\label{CG1}
For a paracompact Hausdorff space $X$ the following games are equivalent.
\begin{enumerate}[wide=0pt,label={\upshape(\arabic*)},leftmargin=*]
  \item $\Guf(\mathcal{O},\Omega)$.
  \item $\SGuf(\mathcal{O},\Omega)$.
  \item $\SSGf(\mathcal{O},\Omega)$.
\end{enumerate}
\end{Cor}

We now show that the hypothesis on the space $X$ in Theorem~\ref{TG11}, Theorem~\ref{TG13} and Corollary~\ref{CG1} cannot be dropped.
\begin{Ex}
\label{E18}
{\emph{Paracompactness in Theorem~\ref{TG11} and metacompactness in Theorem~\ref{TG13} are essential.}}\\
Consider $X=[0,\omega_1)$, the set of all countable ordinals with the order topology. The space $X$ is Tychonoff but not metacompact (and hence not paracompact). Since $X$ is not Lindel\"{o}f, it does not satisfy the Scheepers property and hence TWO has no winning strategy in $\Guf(\mathcal{O},\Omega)$. We claim that TWO has a winning strategy in the games $\SGuf(\mathcal{O},\Omega)$ and $\SSGf(\mathcal{O},\Omega)$ on $X$. It is enough to show that TWO has a winning strategy in $\SSGf(\mathcal{O},\Omega)$. Let us define a strategy $\sigma$ for TWO in $\SSGf(\mathcal{O},\Omega)$ on $X$ as follows. In the $n$th inning, suppose that $\mathcal{U}_n$ is the move of ONE in $\SSGf(\mathcal{O},\Omega)$. Since $X$ is strongly starcompact, there exists a finite set $F_n\subseteq X$ such that $X=St(F_n,\mathcal{U}_n)$. Choose $\sigma(\mathcal{U}_1,\mathcal{U}_2,\ldots,\mathcal{U}_n)=F_n$ as the response of TWO in $\SSGf(\mathcal{O},\Omega)$. This defines a winning strategy $\sigma$ for TWO in $\SSGf(\mathcal{O},\Omega)$ on $X$.
\end{Ex}

%Next we observe that in Corollary~\ref{CG1} paracompactness is required.
\begin{Ex}
\label{E20}
\emph{Paracompactness in Corollary~\ref{CG1} is essential.}\begin{enumerate}[wide=0pt,label={\upshape{\bf (\arabic*)}}]
  \item We construct a Hausdorff metacompact starcompact space $X$ which is not paracompact. Let $\kappa$ be an infinite cardinal and $D=\{d_\alpha : \alpha<\kappa\}$ be the discrete space of cardinality $\kappa$. Let $aD=D\cup\{\infty\}$ be the one point compactification of $D$. In the product space $aD\times(\omega+1)$, replace the local base of the point $(\infty,\omega)$ by the family $\{U\setminus(D\times\{\omega\}) : (\infty,\omega)\in U\;\text{and}\; U\;\text{is an open set in}\; aD\times(\omega+1)\}$. Let $X$ be the space obtained by such replacement.  Observe that TWO has no winning strategy in the game $\SSGf(\mathcal{O},\Omega)$ on $X$ since $X$ is not strongly star-Scheepers. Using starcompactness we can construct a winning strategy for TWO in $\SGuf(\mathcal{O},\Omega)$ on $X$ in line of Example~\ref{E18}.\\

  \item Let $aD$ be the one point compactification of the discrete space $D$ of cardinality $\mathfrak c$ and consider the subspace $X=(aD\times [0,\mathfrak{c}^+))\cup (D\times\{\mathfrak{c}^+\})$ of the product space $aD\times [0,\mathfrak{c}^+]$. The space $X$ obtained such a way is Tychonoff starcompact but not paracompact. In addition, $X$ is not strongly star-Scheepers. Using similar reasoning it can be shown that TWO has a winning strategy in $\SGuf(\mathcal{O},\Omega)$ but TWO has no winning strategy in $\SSGf(\mathcal{O},\Omega)$ on $X$.
\end{enumerate}
\end{Ex}

Also we can remark that the games $\SGuf(\mathcal{O},\Omega)$ and $\Guf(\mathcal{O},\Omega)$ are not equivalent in general.

\begin{Th}
\label{T43}
If $X$ is a space such that ONE does not have a winning strategy in the game $\Guf(\mathcal{O},\Omega)$ on $X$, then each large cover of $X$ is weakly groupable.
\end{Th}
\begin{proof}
Let $\mathcal{U}=\{U_n : n\in\mathbb{N}\}$ be a large cover of $X$. We define a strategy $\sigma$ for ONE in the game $\Guf(\mathcal{O},\Omega)$ on $X$ as follows. Consider $\sigma(\emptyset)=\mathcal{U}$ as the first move of ONE. If TWO responds by selecting a finite subset $\mathcal{V}_1\subseteq\sigma(\emptyset)$, then ONE plays $\sigma(\mathcal{V}_1)=\mathcal{U}\setminus\mathcal{V}_1$ in the second inning. If TWO responds with a finite subset $\mathcal{V}_2\subseteq\sigma(\mathcal{V}_1)$, then ONE plays $\sigma(\mathcal{V}_1,\mathcal{V}_2)=\mathcal{U}\setminus(\mathcal{V}_1\cup\mathcal{V}_2)$ and so on. Thus we get a legitimate strategy $\sigma$ for ONE in $\Guf(\mathcal{O},\Omega)$ on $X$. Since ONE does not have a winning strategy in $\Guf(\mathcal{O},\Omega)$ on $X$, $\sigma$ is not a winning strategy for ONE. Now there is a $\sigma$-play $\sigma(\emptyset),\mathcal{V}_1,\sigma(\mathcal{V}_1), \mathcal{V}_2, \sigma(\mathcal{V}_1,\mathcal{V}_1),\ldots$ which is lost by ONE. It follows that $\{\cup\mathcal{V}_n : n\in\mathbb{N}\}$ is an $\omega$-cover of $X$ and the members of the sequence $(\mathcal{V}_n)$ of moves by TWO are pairwise disjoint. If any member of $\mathcal{U}$ are not present in the sequence $(\mathcal{V}_n)$, then after the construction of the play they can be distributed among $\mathcal{V}_n$'s so that $\mathcal{U}$ is weakly groupable.
\end{proof}

Summarizing Theorem~\ref{T43}, Theorem~\ref{TG11} and Theorem~\ref{TG13}, we obtain the following.
\begin{Cor}
\hfill
\begin{enumerate}[wide=0pt,label={\upshape(\arabic*)},
ref={\theCor(\arabic*)},leftmargin=*]
  \item \label{C24} Let $X$ be a paracompact Hausdorff space. If ONE does not have a winning strategy in the game $\SGuf(\mathcal{O},\Omega)$ on $X$, then each large cover of $X$ is weakly groupable.
  \item \label{C25} Let $X$ be a metacompact space. If ONE does not have a winning strategy in the game $\SSGf(\mathcal{O},\Omega)$ on $X$, then each large cover of $X$ is weakly groupable.
\end{enumerate}
\end{Cor}
%Combining Theorem~\ref{T43} and Theorem~\ref{TG11}, we obtain the following.
%\begin{Cor}
%\label{C24}
%Let $X$ be a paracompact Hausdorff space. If ONE does not have a winning strategy in the game $\SGuf(\mathcal{O},\Omega)$ on $X$, then each large cover of $X$ is weakly groupable.
%\end{Cor}
%
%Again combining Theorem~\ref{T43} and Theorem~\ref{TG13}, we obtain the following.
%\begin{Cor}
%\label{C25}
%Let $X$ be a metacompact space. If ONE does not have a winning strategy in the game $\SSGf(\mathcal{O},\Omega)$ on $X$, then each large cover of $X$ is weakly groupable.
%\end{Cor}
\begin{Def}
An open cover $\mathcal{U}$ of $X$ is said to be star-weakly groupable if it can be expressed as a countable union of finite, pairwise disjoint subfamilies $\mathcal{V}_n$, $n\in\mathbb{N}$, such that for each finite set $F\subseteq X$ we have $F\subseteq St(\cup\mathcal{V}_n,\mathcal{U})$ for some $n$.
%The collection of all star-weakly groupable open covers of $X$ is denoted by $\mathcal{O}^{s\mbox{-}wgp}$.
\end{Def}

We sketch the proof of the next result, which is a star variation of Theorem~\ref{T43}.
\begin{Th}
\label{T42}
If $X$ is a space such that ONE does not have a winning strategy in the game $\SGuf(\mathcal{O},\Omega)$ on $X$, then each countable large cover of $X$ is star-weakly groupable.
\end{Th}
\begin{proof}
Let $\mathcal{U}=\{U_n : n\in\mathbb{N}\}$ be a large cover of $X$. We define a strategy $\sigma$ for ONE in $\SGuf(\mathcal{O},\Omega)$ on $X$ as follows. Consider $\sigma(\emptyset)=\mathcal{U}$ as the first move of ONE. If TWO responds by selecting a finite subset $\mathcal{V}_1\subseteq\sigma(\emptyset)$, then ONE plays $\sigma(\mathcal{V}_1)=\mathcal{U}\setminus\mathcal{V}_1$ in the second inning. If TWO responds with a finite subset $\mathcal{V}_2\subseteq\sigma(\mathcal{V}_1)$, then ONE plays $\sigma(\mathcal{V}_1,\mathcal{V}_2)=\mathcal{U}\setminus(\mathcal{V}_1\cup\mathcal{V}_2)$ and so on. Thus we get a legitimate strategy $\sigma$ for ONE in $\SGuf(\mathcal{O},\Omega)$ on $X$. Since ONE does not have a winning strategy in $\SGuf(\mathcal{O},\Omega)$ on $X$, $\sigma$ is not a winning strategy for ONE. So there exists a $\sigma$-play $\sigma(\emptyset),\mathcal{V}_1,\sigma(\mathcal{V}_1), \mathcal{V}_2, \sigma(\mathcal{V}_1,\mathcal{V}_2),\ldots$ which is lost by ONE. Observe that $\{St(\cup\mathcal{V}_n,\sigma(\mathcal{V}_1,\mathcal{V}_2,\ldots,\mathcal{V}_{n-1})) : n\in\mathbb{N}\}$ is an $\omega$-cover of $X$ and the members of the sequence $(\mathcal{V}_n)$ of moves by TWO are pairwise disjoint. It can be shown that $\mathcal{U}$ is star-weakly groupable.
%If any members of $\mathcal{U}$ are not present in the sequence $(\mathcal{V}_n)$, then after the construction of the play they can be distributed among $\mathcal{V}_n's$ so that the result guarantees that $\mathcal{U}$ is star-weakly groupable.
\end{proof}

\begin{Cor}
\label{C4}
If $X$ is a space such that ONE does not have a winning strategy in the game $\SSGf(\mathcal{O},\Omega)$ on $X$, then each countable large cover of $X$ is star-weakly groupable.
\end{Cor}
%If $\mathcal{A}$ and $\mathcal{B}$ are collections of open covers of a space $X$, then the game $\Gf(\mathcal{A},\mathcal{B})$ on $X$ is played as follows. Players ONE and TWO play an inning for each positive integer $n$. In the $n$th inning ONE chooses an open cover $\mathcal{U}_n$ of $X$ and TWO responds by selecting a finite subset $\mathcal{V}_n\subseteq\mathcal{U}_n$. TWO wins the play $\mathcal{U}_1, \mathcal{V}_1, \mathcal{U}_2, \mathcal{V}_2,\ldots, \mathcal{U}_n, \mathcal{V}_n,\ldots$ of this game if $\cup_{n\in\mathbb{N}}\mathcal{V}_n$ is an open cover of $X$. Otherwise ONE wins.
The next result can be similarly verified.
\begin{Th}
\label{T44}
If ONE does not have a winning strategy in the game $\Guf(\mathcal{O},\Omega)$ on $X$, then ONE does not have a winning strategy in the game $\Gf(\Omega,\Lambda^{wgp})$ on $X$.
\end{Th}

Combining Theorem~\ref{T44}, Theorem~\ref{TG11} and Theorem~\ref{TG13}, we obtain the following.
\begin{Cor}
\hfill
\begin{enumerate}[wide=0pt,label={\upshape(\arabic*)},
ref={\theCor(\arabic*)},leftmargin=*]
  \item \label{C26} Let $X$ be a paracompact Hausdorff space. If ONE does not have a winning strategy in the game $\SGuf(\mathcal{O},\Omega)$ on $X$, then ONE does not have a winning strategy in the game $\Gf(\Omega,\Lambda^{wgp})$ on $X$.
  \item \label{C27} Let $X$ be a metacompact space. If ONE does not have a winning strategy in the game $\SSGf(\mathcal{O},\Omega)$ on $X$, then ONE does not have a winning strategy in the game $\Gf(\Omega,\Lambda^{wgp})$ on $X$.
\end{enumerate}
\end{Cor}
%\begin{Cor}
%\label{C26}
%Let $X$ be a paracompact Hausdorff space. If ONE does not have a winning strategy in the game $\SGuf(\mathcal{O},\Omega)$ on $X$, then ONE does not have a winning strategy in the game $\Gf(\Omega,\Lambda^{wgp})$ on $X$.
%\end{Cor}
%
%\begin{Cor}
%\label{C27}
%Let $X$ be a metacompact space. If ONE does not have a winning strategy in the game $\SSGf(\mathcal{O},\Omega)$ on $X$, then ONE does not have a winning strategy in the game $\Gf(\Omega,\Lambda^{wgp})$ on $X$.
%\end{Cor}

\begin{Th}
\label{T57}
If TWO has a winning strategy in the game $\SGuf(\mathcal{O},\Omega)$ on $X$, then TWO has a winning strategy in the game $\NSGuf(\mathcal{O},\Omega)$ on $X$.
\end{Th}
%\begin{proof}
%Let $\sigma$ be a winning strategy for TWO in the game $\SGuf(\mathcal{O},\Omega)$ on $X$. We define a strategy $\tau$ for TWO in the game $\NSGuf(\mathcal{O},\Omega)$ on $X$ as follows. In the $n$th inning, let $\mathcal{U}_n$ be the move of ONE. TWO responds by selecting a finite subset $\sigma(\mathcal{U}_1,\mathcal{U}_2,\ldots,\mathcal{U}_n)\subseteq\mathcal{U}_n$ in $\SGuf(\mathcal{O},\Omega)$. Consider $\tau(\mathcal{U}_1,\mathcal{U}_2,\ldots,\mathcal{U}_n)=
%\sigma(\mathcal{U}_1,\mathcal{U}_2,\ldots,\mathcal{U}_n)$ as the response of TWO in $\NSGuf(\mathcal{O},\Omega)$. Thus we get a strategy $\tau$ for TWO in the game $\NSGuf(\mathcal{O},\Omega)$ on $X$. To show that $\tau$ is a winning strategy we pick a $\tau$-play $\mathcal{U}_1,\tau(\mathcal{U}_1),\mathcal{U}_2,\tau(\mathcal{U}_1,\mathcal{U}_2),\ldots$. The corresponding $\sigma$-play is $\mathcal{U}_1,\sigma(U_1),\mathcal{U}_2,\sigma(\mathcal{U}_1,\mathcal{U}_2),\ldots$ and since $\sigma$ is a winning strategy for TWO in $\SGuf(\mathcal{O},\Omega)$, $\{St(\cup\sigma(\mathcal{U}_1,\mathcal{U}_2,\ldots,\mathcal{U}_n),\mathcal{U}_n) : n\in\mathbb{N}\}$ is an $\omega$-cover of $X$. It follows that $\{St(\cup_{m\in\mathbb{N}}(\cup\tau(\mathcal{U}_1,\mathcal{U}_2,\ldots,\mathcal{U}_m)),\mathcal{U}_n) : n\in\mathbb{N}\}$ is also an $\omega$-cover of $X$ and hence $\tau$ is a winning strategy for TWO in the game $\NSGuf(\mathcal{O},\Omega)$ on $X$.
%\end{proof}

\begin{Th}
\label{T58}
If ONE has a winning strategy in the game $\NSGuf(\mathcal{O},\Omega)$ on $X$, then ONE has a winning strategy in the game $\SGuf(\mathcal{O},\Omega)$ on $X$.
\end{Th}

%The following example shows that the converse of Theorem~\ref{T57} and Theorem~\ref{T58} is not true.
\begin{Ex}
\label{E21}
\emph{Converse of Theorem~\ref{T57} and Theorem~\ref{T58} is not true.}\\
The space $X=\Psi(\mathcal{A})$ where $|\mathcal{A}|=\mathfrak{c}$ is Tychonoff and strongly star-Lindel\"{o}f but not star-Scheepers. Thus TWO has no winning strategy and ONE has a winning strategy in $\SGuf(\mathcal{O},\Omega)$ on $X$. We claim that TWO has a winning strategy in $\NSGuf(\mathcal{O},\Omega)$ on $X$.  Define a strategy $\sigma$ for TWO in $\NSGuf(\mathcal{O},\Omega)$ on $X$ as follows. Let $\mathcal{U}_1$ be the first move of ONE. Since $X$ is strongly star-Lindel\"{o}f, choose a countable subset $A=\{x_n : n\in\mathbb{N}\}$ of $X$ such that $St(A,\mathcal{U}_1)=X$. Next choose a $U_1\in\mathcal{U}_1$ such that $x_1\in U_1$ and consider $\sigma(\mathcal{U}_1)=\{U_1\}$ as the response of TWO. Let $\mathcal{U}_2$ be the second move of ONE. Also there is a $U_2\in\mathcal{U}_2$ such that $x_2\in U_2$ and consider $\sigma(\mathcal{U}_1,\mathcal{U}_2)=\{U_2\}$ as the response of TWO. Suppose that $\mathcal{U}_3$ is the third move of ONE and so on. Thus we define a winning strategy $\sigma$ for TWO in the game $\NSGuf(\mathcal{O},\Omega)$ on $X$.

We now show that ONE does not have a winning strategy in $\NSGuf(\mathcal{O},\Omega)$ on $X$. Let $\tau$ be a strategy for ONE in $\NSGuf(\mathcal{O},\Omega)$. We construct a $\tau$-play as follows. Suppose that $\tau(\emptyset)=\mathcal{U}$ is the first move of ONE. Again by the strongly star-Lindel\"{o}f property of $X$, there is a countable subset $A=\{x_n : n\in\mathbb{N}\}$ of $X$ such that $St(A,\mathcal{U})=X$. Choose a $U_1\in\mathcal{U}$ such that $x_1\in U_1$ and define $\mathcal{V}_1=\{U_1\}$ as the response of TWO. The second move of ONE is $\tau(\mathcal{V}_1)$ and subsequently we can find a $U_2\in\tau(\mathcal{V}_1)$ such that $x_2\in U_2$. Define $\mathcal{V}_2=\{U_2\}$ as the response of TWO. The third move of ONE is $\tau(\mathcal{V}_2)$ and so on. This defines a $\tau$-play in $\NSGuf(\mathcal{O},\Omega)$. It can be seen that the $\tau$-play is lost by ONE and hence ONE does not have a winning strategy in the game $\NSGuf(\mathcal{O},\Omega)$ on $X$.
\end{Ex}

The relation between the winning strategies of the players ONE and TWO in the games (for any space $X$) considered here can be outlined into the following diagram (Figure~\ref{dig3}), where each of the implications
\begin{center}
\begin{tikzcd}
G\arrow[r,->,line width=.022cm]&[-.3cm] H,\end{tikzcd} \begin{tikzcd}
G\arrow[dashed,r,->,line width=.022cm]&[-.3cm] H\end{tikzcd} and \begin{tikzcd}
G\arrow[r,dotted,->,line width=.022cm] & [-.3cm]H\end{tikzcd}
\end{center}
holds if winning strategies for TWO in $G$ produce winning strategies for TWO in $H$ as well as winning strategies for ONE in $H$ produce winning strategies for ONE in $G$ and the selection principle for $G$ implies the selection principle for $H$.

\begin{figure}[h!]
\begin{adjustbox}{max width=\textwidth,max height=\textheight,keepaspectratio,center}
\begin{tikzcd}[column sep=2.5cm,row sep=2.5cm,arrows={thick}]
%level 4
\NSG1(\mathcal{O},\Gamma)\arrow[r]&
\NSGuf(\mathcal{O},\Gamma)\arrow[r]&
\NSGuf(\mathcal{O},\Omega)\arrow[r]&
\NSGuf(\mathcal{O},\mathcal{O})&
\NSG1(\mathcal{O},\mathcal{O})\arrow[l]
\\
%level 3
\SG1(\mathcal{O},\Gamma)\arrow[r]\arrow[d]\arrow[u]&
\SGuf(\mathcal{O},\Gamma)\arrow[r]\arrow[u]&
\SGuf(\mathcal{O},\Omega)\arrow[r]\arrow[u]&
\SGf(\mathcal{O},\mathcal{O})\arrow[u]&
\SG1(\mathcal{O},\mathcal{O})\arrow[l]\arrow[u]
\\
%level 2
\SG1(\mathcal{O},\Omega)\arrow[d]\arrow[dotted,to path={([xshift=-2.5cm,yshift=-1.5cm]\tikztotarget.south)
-- (\tikztotarget)}]{rru}\arrow[dotted,to path={-- ([xshift=1.6cm]\tikztostart.east)|-([xshift=-2cm,yshift=-1.5cm]\tikztotarget.south) -- (\tikztotarget)}]{rrrru}&
\SSGf(\mathcal{O},\Gamma)\arrow[r]\arrow[u]&
\SSGf(\mathcal{O},\Omega)\arrow[r]\arrow[u]&
\SSGf(\mathcal{O},\mathcal{O})\arrow[u]&
\SSG1(\mathcal{O},\mathcal{O})\arrow[l]\arrow[u]
\\
%level 1
\NSG1(\mathcal{O},\Omega)\arrow[dashed,to path={-- ([xshift=1.6cm]\tikztostart.east)|-([xshift=-2cm,yshift=-7.5cm]\tikztotarget.south)
--([xshift=-2cm,yshift=-1.5cm]\tikztotarget.south) -- (\tikztotarget)}]{rrrruuu}\arrow[dashed,to path={([xshift=-2.5cm,yshift=-7.5cm]\tikztotarget.south)
--([xshift=-2.5cm,yshift=-1.5cm]\tikztotarget.south) -- (\tikztotarget)}]{rruuu}&
\Guf(\mathcal{O},\Gamma)\arrow[r]\arrow[u]&
\Guf(\mathcal{O},\Omega)\arrow[r]\arrow[u]&
\Gf(\mathcal{O},\mathcal{O})\arrow[u]&
\G1(\mathcal{O},\mathcal{O})\arrow[l]\arrow[u]
\end{tikzcd}
\end{adjustbox}
\caption{Diagram of winning strategies}
\label{dig3}
\end{figure}
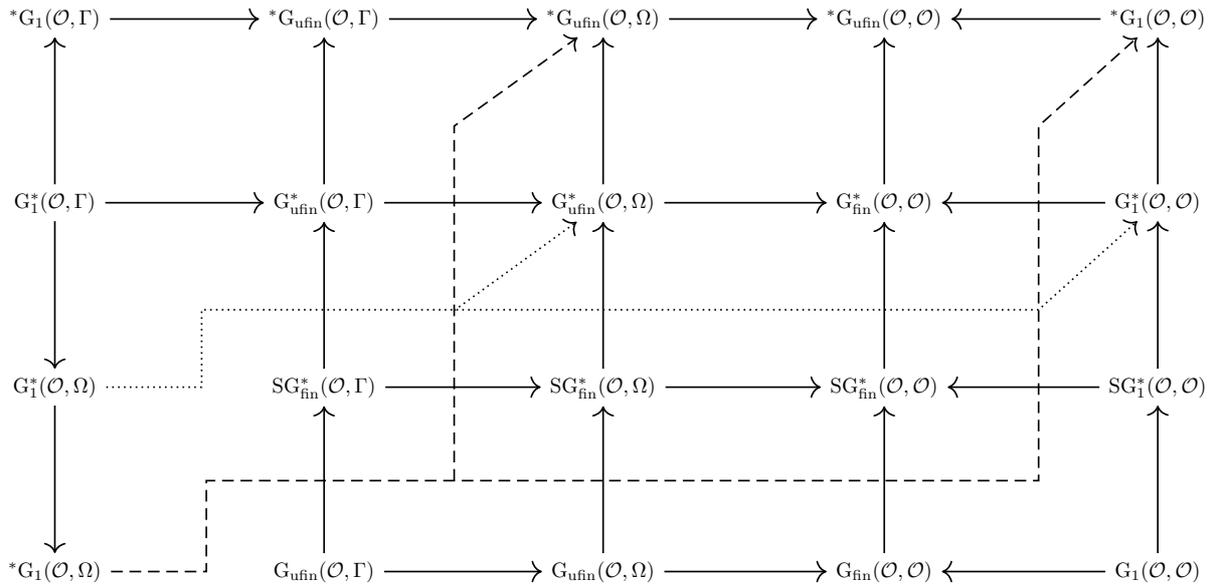

\section{Concluding Remarks}
We surmise a result related to the strongly star-Scheepers property. See \cite[Theorem 3.24]{SSSP} for similar investigation in the context of strongly star-Menger spaces.
\begin{Conj}
\label{T33}
Let $X$ be a strongly star-Lindel\"{o}f space of the form $Y\cup Z$, where $Y$ is a closed discrete set and $Z$ is a $\sigma$-compact subset of $X$ with $Y\cap Z=\emptyset$. If $X$ is strongly star-Scheepers, then ONE does not have a winning strategy in the game $\SSGf(\mathcal{O},\Omega)$ on $X$.
\end{Conj}

Considering the above conjecture is true, we obtain the amusing game theoretic observation.
\begin{Res}
If $X$ is $\Psi$-space or the Niemytzki plane, then $X$ is strongly star-Scheepers if and only if ONE does not have a winning strategy in the game $\SSGf(\mathcal{O},\Omega)$ on $X$.
\end{Res}
%\begin{Cor}
%\label{C21}
%If $X$ is $\Psi$-space or the Niemytzki plane, then $X$ is strongly star-Scheepers if and only if ONE does not have a winning strategy in the game $\SSGf(\mathcal{O},\Omega)$ on $X$.
%\end{Cor}
%%%%%%%%%%%%%%%%%%%%%%%%%%%%%%%%%%%%%%%%
%%%%%%%%%%%%%%%%%%%%%%%%%%%%%%%%%%%%%%%%
%%%%%%%%%%%%%%%%%%%%%%%%%%%%%%%%%%%%%%%%
%%%%%%%%%%%%%%%%%%%%%%%%%%%%%%%%%%%%%%%%

Next we use the idea of \cite{dcna22,Restricted} to define restricted Scheepers game $\RGuf(\mathcal{O},\Omega)$, restricted star-Scheepers game $\SRGuf(\mathcal{O},\Omega)$ and restricted strongly star-Scheepers game $\SSRGf(\mathcal{O},\Omega)$ on a space $X$.

\noindent $\RGuf(\mathcal{O},\Omega)$: Players ONE and TWO play an inning per each positive integer. At the start of $n$th inning TWO makes an initial move which must be a positive integer $k_n$. ONE then play an open cover $\mathcal{U}_n$ of $X$ and TWO responds by selecting a finite set $\mathcal{V}_n\subseteq\mathcal{U}_n$ with $|\mathcal{V}_n|=k_n$. TWO wins the play if and only if $\{\cup\mathcal{V}_n : n\in\mathbb{N}\}$ is an $\omega$-cover of $X$; otherwise ONE wins. The game $\SRGuf(\mathcal{O},\Omega)$ can be defined analogously, where we demand that $\{St(\cup\mathcal{V}_n,\mathcal{U}_n) : n\in\mathbb{N}\}$ instead of $\{\cup\mathcal{V}_n : n\in\mathbb{N}\}$ is an $\omega$-cover of $X$.\\

\noindent $\SSRGf(\mathcal{O},\Omega)$: Players ONE and TWO play an inning per each positive integer. At the start of $n$th inning TWO makes an initial move which must be a positive integer $k_n$. ONE then play an open cover $\mathcal{U}_n$ of $X$ and TWO responds by selecting a finite set $F_n\subseteq X$ with $|F_n|=k_n$. TWO wins the play if and only if $\{St(F_n,\mathcal{U}_n) : n\in\mathbb{N}\}$ is an $\omega$-cover of $X$; otherwise ONE wins.\\

%\noindent $\SRGuf(\mathcal{O},\Omega)$: Players ONE and TWO play an inning per each positive integer. At the start of $n$th inning TWO makes an initial move which must be a positive integer $k_n$. ONE then play an open cover $\mathcal{U}_n$ of $X$ and TWO responds by selecting a finite set $\mathcal{V}_n\subseteq\mathcal{U}_n$ with $|\mathcal{V}_n|=k_n$. TWO wins the play if and only if $\{St(\cup\mathcal{V}_n,\mathcal{U}_n) : n\in\mathbb{N}\}$ is an $\omega$-cover of $X$; otherwise ONE wins.
It can be seen that the games $\RGuf(\mathcal{O},\Omega)$, $\SRGuf(\mathcal{O},\Omega)$ and $\SSRGf(\mathcal{O},\Omega)$ are all equivalent for paracompact Hausdorff spaces. Similar such investigations may be carried out in line of \cite{dcna22}.

In another note, assume that $\mathfrak{b}<\mathfrak{d}$. By \cite[Proposition]{SCPP} and Corollary~\ref{C1201}, $\Psi(\mathcal{A})$ with $|\mathcal{A}|=\mathfrak{b}$ is strongly star-Scheepers but neither Scheepers nor strongly star-Hurewicz. Thus it is interesting to ask the following questions.

\begin{Prob}
\label{Problem1}
 Does there exist a star-Scheepers space which is neither Scheepers nor star-Hurewicz?
\end{Prob}

\begin{Prob}
\label{Problem2}
  Does there exist a strongly star-Menger space which is neither Menger nor strongly star-Scheepers?
\end{Prob}

\begin{Prob}
\label{Problem3}
  Does there exist a star-Menger space which is neither Menger nor star-Scheepers?
\end{Prob}

\end{document}